\newif\ifIEEE

\documentclass[conference,10pt]{IEEEtran} \IEEEtrue
\newtheorem{definition}{Definition}
\newtheorem{theorem}{Theorem}
\newtheorem{proposition}{Proposition}
\newtheorem{example}{Example}
\newtheorem{remark}{Remark}
\newenvironment{proof}{\begin{IEEEproof}}{\end{IEEEproof}}

\usepackage[T1]{fontenc}
\usepackage{latexsym,bm}
\usepackage{dsfont}
\usepackage{amssymb}
\usepackage[tbtags]{amsmath}

\usepackage[ruled,linesnumbered,lined,boxed,commentsnumbered]{algorithm2e}
\usepackage{color}
\usepackage{pdfsync}
\usepackage{subcaption}
\usepackage{graphicx}
\usepackage{epstopdf}
\usepackage{bibentry}

\usepackage{multirow}
\usepackage{umoline}

\usepackage{todonotes,xspace,afterpage,cancel}
\usepackage[normalem]{ulem}

\usepackage[inline]{enumitem}
\newenvironment{enumingen}[5]{\begin{enumerate*}[label=#5),itemjoin={{#2 }},itemjoin*={{#2 #3 }},after={#4},#1]}{\end{enumerate*}}
\newenvironment{enumin}[2][]{\begin{enumingen}{#1}{#2}{and}{.}{\arabic*}}{\end{enumingen}}
\newenvironment{enuminNoAnd}[2][]{\begin{enumingen}{#1}{#2}{}{.}{\arabic*}}{\end{enumingen}}
\newenvironment{enuminii}[2][]{\begin{enumingen}{#1}{#2}{and}{.}{\roman*}}{\end{enumingen}}

\usepackage{hyperref}

\newcommand{\Init}{\ensuremath{\mathit{X_0}}\xspace}
\newcommand{\Uns}{\ensuremath{\mathit{X_U}}\xspace}
\newcommand{\Prog}{\ensuremath{\mathit{Prog}}\xspace}
\newcommand{\class}{\ensuremath{\mathbf{Class}}\xspace}
\newcommand{\Groba}{Gr{\"o}bner basis\xspace}
\newcommand{\K}{\ensuremath{\mathbb{K}}\xspace}
\newcommand{\R}{\ensuremath{\mathbb{R}}\xspace}

\newcommand{\C}{\ensuremath{\mathbb{C}}\xspace}
\newcommand{\N}{\ensuremath{\mathds{N}}\xspace}
\newcommand{\V}{\ensuremath{\mathds{V}}\xspace}
\newcommand{\emphii}[1]{\textbf{#1}} 
\newcommand{\emphiii}[1]{\textbf{#1}} 
\newcommand{\ideal}[1]{\ensuremath{\langle #1 \rangle}\xspace}
\newcommand{\consys}{\ensuremath{\langle X, \vec{f}, \Init\rangle}\xspace}
\newcommand{\CST}{\ensuremath{\mathcal{C}}\xspace}
\newcommand{\SOS}{\textit{SOS}\xspace}
\newcommand{\sys}{\ensuremath{M}\xspace}
\newcommand{\clus}{\ensuremath{C}\xspace}
\newcommand{\HInit}{\ensuremath{\operatorname{Init}}\xspace}
\newcommand{\HUns}{\ensuremath{\operatorname{Uns}}\xspace}

\newcommand{\vect}[1]{\ensuremath{\vec{#1}}\xspace}
\newcommand{\vx}{\ensuremath{\vect{x}}\xspace}
\newcommand{\vu}{\ensuremath{\vect{u}}\xspace}
\newcommand{\vz}{\ensuremath{\vect{0}}\xspace}

\newcommand{\defeq}{\stackrel{\text{def}}{=}\xspace}
\newcommand{\Lf}{\ensuremath{\mathcal{L}_{\vect{f}}}\xspace}
\newcommand{\Lfg}{\ensuremath{\Lf g}\xspace}

\begin{document}

\ifIEEE\else{
\mainmatter
}\fi

\title{Invariant Clusters for Hybrid Systems}

\ifIEEE\else{
\titlerunning{}
}\fi



\author{
\IEEEauthorblockN{
 Hui Kong\IEEEauthorrefmark{1},
 Sergiy Bogomolov\IEEEauthorrefmark{1},
 Christian Schilling\IEEEauthorrefmark{2},
 Yu Jiang\IEEEauthorrefmark{3},
 Thomas A. Henzinger\IEEEauthorrefmark{1}
}
\IEEEauthorblockA{
 \IEEEauthorrefmark{1}IST Austria, Klosterneuburg, Austria
}
\IEEEauthorblockA{
 \IEEEauthorrefmark{2}University of Freiburg, Freiburg, Germany
}
\IEEEauthorblockA{
 \IEEEauthorrefmark{3}University of Illinois at Urbana-Champaign, Illinois, USA
}
}

\maketitle

\vspace*{-5mm}

\begin{abstract}
In this paper, we propose an approach to automatically compute invariant clusters for semialgebraic hybrid systems. An invariant cluster for an ordinary differential equation (ODE) is a multivariate polynomial invariant $g(\vect{u},\vect{x})=0$, parametric in $\vect{u}$, which can yield an infinite number of concrete invariants by assigning different values to $\vect{u}$ so that every trajectory of the system can be overapproximated precisely by a union of concrete invariants. For semialgebraic systems, which involve ODEs with multivariate polynomial vector flow, invariant clusters can be obtained by first computing the remainder of the Lie derivative of a template multivariate polynomial w.r.t. its \Groba and then solving the system of polynomial equations obtained from the coefficients of the remainder. Based on invariant clusters and sum-of-squares (\SOS) programming, we present a new method for the safety verification of hybrid systems. Experiments on nonlinear benchmark systems from biology and control theory show that our approach is effective and efficient.

\end{abstract} 

\section{Introduction}
Hybrid systems~\cite{henzinger1996theory} are models for systems with interacting discrete and continuous dynamics. Safety verification is among the most challenging problems in verifying hybrid systems, asking whether a set of bad states can be reached from a set of initial states. The safety verification problem for systems described by nonlinear differential equations is particularly complicated because computing the exact reachable set is usually intractable. Existing approaches are mainly based on approximate reachable set computations~\cite{asarin2003reachability,bogomolov-et-al:sttt-2015,bogomolov-etal:cav2012} and abstraction~\cite{tiwari2008abstractions,alur2003progress,bogomolov-etal:hvc2014,DBLP:conf/hvc/BogomolovSBBKG15,bogomolov2014quasi}.

An invariant is a special kind of overapproximation for the reachable set of a system. Since invariants do not involve direct computation of the reachable set, they are especially suitable for dealing with nonlinear hybrid systems. However, automatically and efficiently generating sufficiently strong invariants is challenging on its own.

In this work, we propose an approach to automatically compute invariant clusters for a class of nonlinear semialgebraic systems whose trajectories are algebraic, i.e., every trajectory of the system is essentially a common zero set (algebraic variety) of a set of multivariate polynomial equations. An invariant cluster for a semialgebraic system is a parameterized multivariate polynomial invariant $g(\vect{u},\vect{x})=0$, with parameter $\vect{u}$, which can yield an infinite number of concrete invariants by assigning different values to $\vect{u}$ so that every trajectory of the system can be overapproximated precisely by a union of concrete invariants.

The basic idea of computing invariant clusters is as follows. A sufficient condition for a trajectory of a semialgebraic system to start from and to always stay in the solution set of a multivariate polynomial equation $g(\vect{x})=0$ is that the Lie derivative of $g(\vect{x})$ w.r.t. $\vect{f}$ belongs to the ideal generated by $g(\vect{x})$, i.e., $\Lfg \in \ideal{g(\vect{x})}$, where $\vect{f}$ is the vector flow of the system. Therefore, if some $g(\vect{x})$ satisfies this condition, $g(\vect{x})=0$ is an invariant of the system. Then, according to \Groba theory, $\Lfg \in \ideal{g(\vect{x})}$ implies that the remainder of $\Lfg$ w.r.t. $\ideal{g(\vect{x})}$ must be identical to~$0$. Based on this theory, we first set up a template polynomial $g(\vect{u},\vect{x})$ and then compute the remainder $r(\vect{u},\vect{x})$ of $\Lfg$ w.r.t. $\ideal{g(\vect{u},\vect{x})}$. Since $r(\vect{u},\vect{x})\equiv 0$ implies that all coefficients $a_i(\vect{u})$ of $\vect{x}$ in $r(\vect{u},\vect{x})$ are equal to zero, we can set up a system $P$ of polynomial equations on $\vect{u}$ from the coefficients $a_i(\vect{u})$. By solving $P$ we get a set \CST of constraints on $\vect{u}$. For those elements in \CST that are linear in $\vect{u}$, the corresponding parameterized polynomial equations $g(\vect{u},\vect{x})=0$ form invariant clusters. Based on invariant clusters and \SOS programming, we propose a new method for the safety verification of hybrid systems.

\smallskip

The main contributions of this paper are as follows:
\begin{enuminNoAnd}{.}
\item We propose to generate invariant clusters for semialgebraic systems based on computing the remainder of the Lie derivative of a template polynomial w.r.t. its \Groba and solving the system of polynomial equations obtained from the coefficients of the remainder. Our approach avoids  \Groba computation and first-order quantifier elimination

\item We present a method to overapproximate trajectories precisely by using invariant clusters

\item We apply invariant clusters to the safety verification of semialgebraic hybrid systems based on \SOS programming

\item We implemented a prototype tool to perform the aforementioned steps automatically. Experiments show that our approach is effective and efficient
\end{enuminNoAnd}

\smallskip

The paper is organized as follows. Section~\ref{sec:preliminaries} is devoted to the preliminaries. In Section~\ref{sec:computation}, we introduce the approach to computing invariant clusters and using them to characterize trajectories. In Section~\ref{sec:verification}, we present a method to verify safety properties for semialgebraic continuous and hybrid systems based on invariant clusters. In Section~\ref{sec:evaluation}, we present our experimental results. In Section~\ref{sec:relatedwork}, we introduce some related works. Finally, we conclude our paper in Section~\ref{sec:conclusion}.

\section{Preliminaries}\label{sec:preliminaries}
In this section, we recall some backgrounds we need throughout the paper. We first clarify some notation conventions. If not specified otherwise, we decorate vectors $\vect{\cdot}$, we use the symbol $\K$ for a field, $\R$ for the real number field, $\C$ for the complex number field (which is algebraically closed) and $\N$ for the set of natural numbers, and all the polynomials involved are multivariate polynomials. In addition, for all the polynomials $g(\vu,\vx)$, we denote by $\vu$ the vector composed of all the $u_i$ and denote by $\vx$ the vector composed of all the remaining variables that occur in the polynomial.

\begin{definition}[Ideal]\label{def_ideal}\cite{cox2010ideals}
A subset $I$ of $\K[\vect{x}]$, is called an \emphii{ideal} if
\begin{enumin}{;}
  \item $0 \in I$
  \item if $p,q \in I$, then $p + q\in I$
  \item if $p\in I$ and $q\in \K[\vect{x}]$, then $pq\in I$
\end{enumin}
\end{definition}

\begin{definition}[Generated ideal]\cite{cox2010ideals}
  Let $g_1,\dots,g_s$ be polynomials in $\K[\vect{x}]$. The \emphii{ideal generated by $\{g_1,\dots,g_s\}$} is
  \begin{equation*}
    \ideal{g_1,\dots,g_s} \defeq \big\{ \sum_{i=1}^s h_ig_i \mid h_1,\dots,h_s\in \K[\vect{x}]\big\}.
  \end{equation*}
\end{definition}

\begin{definition}[Algebraic variety]\label{def_VarIdea}
Let $\K$ be an algebraically closed field and $I\subset \K[\vect{x}]$ be an ideal. We define the \emphii{algebraic variety} of $I$ as
\begin{equation*}
 \mathds{V}(I) \defeq \{\vect{x} \in \K^n  \mid  f(\vect{x})= 0 \text{ for 
 } f\in I\}.
\end{equation*}
\end{definition}

%
%
%

Next, we present the notation of the Lie derivative, which is widely used in the discipline of differential geometry.

\begin{definition}[Lie derivative]\label{def_Lie}
 For a given polynomial $p\in \K[\vect{x}]$ over $\vect{x}=(x_1,\cdots,x_n)$ and a continuous system $\dot{\vect{x}} = \vect{f}$, where $\vect{f}=(f_1,\dots,f_n)$, the \emphii{Lie derivative} of $p\in \K[\vect{x}]$ along $f$ of order $k$ is defined as follows.
 \begin{equation}
		\Lf^{k} p \stackrel{\text{def}}{=}
		\left\{\begin{array}{lc} p, & k = 0\\
      \sum_{i=1}^n \frac{\partial{\Lf^{k-1} p}}{\partial{x_i}} \cdot f_i, & k \geq 1
    \end{array}\right.
 \end{equation}
\end{definition}

Essentially, the $k$-th order Lie derivative of $p$ is the $k$-th derivative of $p$ w.r.t. time, i.e., reflects the change of $p$ over time.
We write $\Lf p$ for $\Lf^{1} p$.

%

We furthermore use the following theorem for deciding the existence of a real solution of a system of polynomial constraints.

\begin{theorem}[Real Nullstellensatz]\cite{stengle1974nullstellensatz}\label{citedthrm1}
 The system of multivariate polynomial equations and inequalities
  $p_1(\vect{x})=0$, $\dots$, $p_{m_1}(\vect{x})=0$, \ \
  $q_1(\vect{x})\geq 0$, $\dots$, $q_{m_2}(\vect{x})\geq 0$, \ \
  $r_1(\vect{x})>0$, $\dots$, $r_{m_3}(\vect{x})>0$
  either has a solution in $\R^n$, or there exists the following polynomial identity
  \begin{equation}\label{thrm1eq}
    \begin{gathered}
     \sum_{i=1}^{m_1} {}\beta_i p_i \ + \sum_{v\in \{0,1\}^{m_2}} \big( \sum_t b_{tv} \big)^2\cdot\prod_{j=1}^{m_2} q_j^{v_j} + \prod_{k=1}^{m_3}r_k^{d_k} \\
     + \sum_{v\in \{0,1\}^{m_3}} \big( \sum_t c_{tv} \big)^2\cdot \prod_{k=1}^{m_3} r_k^{v_k} + \sum_w s_w^2  = 0
    \end{gathered}
  \end{equation}
  where $d_k\in \N$ and $p_i, q_j, r_k, \beta_j, b_{tv}, c_{tv},s_w$ are polynomials in $\R[\vect{x}]$.
\end{theorem}

\begin{remark}\label{remark_sos}
Theorem~\ref{citedthrm1} enables us to efficiently decide if a system of polynomial equations and inequalities has a real solution. By moving the term $\sum_w s_w^2$ in equation~\eqref{thrm1eq} to the right-hand side and denoting the remaining terms by $R(\vect{x},\vect{y})$, we have $-R(\vect{x},\vect{y}) = \sum_w s_w^2$, which means that $-R(\vect{x},\vect{y})$ is a sum-of-squares. Therefore, finding a set of polynomials $\beta_j, r_k, b_{tv}, c_{tv}, s_w$ as well as some $d_k$'s that make $-R(\vect{x},\vect{y})$ a sum-of-squares is sufficient to prove that the system has no real solution, which can be done efficiently by \SOS programming~\cite{prajna2005sostools}.
\end{remark}

In this paper, we focus on semialgebraic continuous and hybrid systems which are defined in the following, respectively.

\begin{definition}[Semialgebraic system]
A \emphii{semialgebraic system} is a tuple $\sys \stackrel{\text{def}}{=} \consys$, where
\begin{enumerate}
  \item $X$ is the state space of the system $\sys$,
  \item $\vect{f}$ is a Lipschitz continuous vector flow function, and
  \item \Init is the initial set described by a semialgebraic set.
\end{enumerate}
\end{definition}

The Lipschitz continuity guarantees existence and uniqueness of the differential equation $\dot{\vect{x}}=\vect{f}$. A trajectory of a semialgebraic system is defined as follows.

\begin{definition}[Trajectory]
Let $\sys$ be a semialgebraic system. A \emphii{trajectory} originating from a point $\vect{x}_0\in \Init$ to time $T>0$ is a continuous and differentiable function $\vect{x}(t):[0, T)\to \mathbb{R}^n$ such that
  \begin{enuminii}{,}
   \item $\vect{x}(0)=\vect{x}_0$
   \item $\forall \tau \in [0,T): \frac{d\vect{x}}{dt}\big|_{t=\tau} = \vect{f}(\vect{x}(\tau))$
  \end{enuminii}
\end{definition}

\vspace*{1mm}

\begin{definition}[Safety]
Given an unsafe set $\Uns$, a semialgebraic system $\sys = \consys$ is said to be \emphii{safe} if no trajectory $\vect{x}(t)$ of $\sys$ satisfies both $\vect{x}(0)\in \Init$ and $\exists \tau\in \mathbb{R}_{\geq 0}:\vect{x}(\tau)\in \Uns$.
\end{definition}

\begin{definition}[Hybrid System] A \emphii{hybrid system} is described by a tuple $\mathcal{H} \defeq \langle L,X,E,R,G,I,F\rangle$, where
\begin{itemize}
\item $L$ is a finite set of locations (or modes);
\item $X \subseteq \mathbb{R}^n$ is the continuous state space. The hybrid state space of the system is denoted by $\mathcal{X} = L \times X$  and a state is denoted by $(l,\vx) \in \mathcal{X}$;
\item $E \subseteq L \times L$ is a set of discrete transitions;
\item $G:E \to 2^X$ is a guard mapping over discrete transitions;
\item $R:E \times 2^X \to 2^X$ is a reset mapping over discrete transitions;
\item $I : L \to 2^X$ is an invariant mapping;
\item $F : L \to (X \to X)$ is a vector field mapping which assigns to each location $l$ a vector field $\vect{f}$.
\end{itemize}
\end{definition}

The transition and dynamic structure of the hybrid system defines a set of trajectories. A trajectory is a sequence originating from a state $(l_0,\vx_0) \in \mathcal{X}_0$, where $\mathcal{X}_0 \subseteq \mathcal{X}$ is an initial set, and consisting of a series of interleaved continuous flows and discrete transitions. During the continuous flows, the system evolves following the vector field $F(l)$ at some location $l\in L$ as long as the invariant condition $I(l)$ is not violated. At some state $(l,\vx)$, if there is a discrete transition $(l,l')\in E$ such that $(l,\vx) \in G(l,l')$ (we write $G(l,l')$ for $G((l,l'))$), then the discrete transition can be taken and the system state can be reset to $R(l,l',\vx)$. The problem of safety verification of a hybrid system is to prove that the hybrid system cannot reach an unsafe set $\Uns$ from an initial set $\Init$.

\section{Computation of Invariant Clusters}\label{sec:computation}
In this section, we first introduce the notion of invariant cluster and then show how to compute a set of invariant clusters and how to use it to represent every trajectory of a semialgebraic system.

\subsection{Foundations of invariants and invariant clusters}\label{subsec:invcondition}
Given a semialgebraic system $\sys$, for any trajectory $\vx(t)$ of $\sys$, if we can find a multivariate polynomial $g(\vx)\in \R[\vx]$ such that $g(\vx(0)) \sim 0$ implies $g(\vx(t))\sim 0$ for all $t>0$, where $\sim\ \in \{<, \leq, =, \geq, >\}$, then $g(\vx) \sim 0$ is an \emph{invariant} of the system. For the convenience of presentation, we call $g(\vx)$ an \emph{invariant polynomial} of $\sys$. In addition, a trajectory $\vx(t)$ is said to be \emph{algebraic} if there exists an invariant $g(\vx) = 0$ for $\vx(t)$ and $g(\vx) \not\equiv 0$. In the following, we present a sufficient condition for a polynomial $g(\vx)$ to be an invariant polynomial.

\begin{proposition}\label{prop9}
Let $\sys=\consys$ be a semialgebraic system and $g(\vx)\in \R[\vx]$. Then $g \sim 0$ is an invariant of $\sys$ for every $\sim\ \in \{<, \leq, =, \geq, >\}$ if $g(\vx)$ satisfies
\begin{equation}\label{eq10}
  \Lfg \in \ideal{g}
\end{equation}
\end{proposition}
\begin{proof}
See Appendix~\ref{prfofprop9}.
\end{proof}

Proposition~\ref{prop9} states that all the polynomial equations and inequalities $g \sim 0$ are invariants of $\sys$ if the Lie derivative of $g$ belongs to the ideal \ideal{g}. Note that every invariant satisfying condition~\eqref{eq10} defines a closed set for trajectories, that is, no trajectory can enter or leave the set defined by the invariant.

\smallskip

For a semialgebraic system whose trajectories are algebraic, the trajectories can usually be divided into several groups and in each group all trajectories show similar curves. Essentially, these similar curves can be described identically by a unique parameterized polynomial equation, which we characterize as an invariant cluster. The computation method of invariant clusters is presented in Subsection~\ref{sec_compute}.

\begin{definition}[Invariant cluster]\label{def_invcluster}
An \emphii{invariant cluster} $\clus$ of a semialgebraic system is a set of invariants which can be uniformly described as
$
 \clus=\{g(\vu,\vx)=0 \mid \vu\in \R^K\setminus\{\vec{0}\} \}
$,
where $g(\vu,\vx)=\sum_{i=1}^M c_i(\vu) X^i$ satisfies $\Lfg \in \ideal{g}$ and $c_i(\vu)\in \R[\vu]$ are fixed linear polynomials on $\vu=(u_1,\cdots, u_K)$, $X^i$ are monomials on $\vx=(x_1,\cdots,x_n)$, and $M, K\in \N$.
\end{definition}

Note that by requiring $\vu\neq \vect{0}$ in Definition~\ref{def_invcluster} and the following related definitions, we exclude the trivial invariant $0=0$. Given an invariant cluster, by varying the parameter $\vu$, we may obtain an infinite set of concrete invariants for the system. To be intuitive, we present a running example to demonstrate the related concepts throughout the paper.

\begin{example}[running example]\label{exam1}
Consider the semialgebraic system $\sys_1$ described by $\left[ \dot{x}, \dot{y} \right] = \left[ y^2, xy \right]$. The set $C^*=\{u_1 - u_3(x^2 - y^2)=0\mid (u_1,u_3)\in \R^2 \setminus \{\vect{0}\}\}$ is an invariant cluster. It is easy to verify that the polynomial $u_1 - u_3(x^2 - y^2)$ satisfies condition~\eqref{eq10} for all $(u_1,u_3)\in \R^2$.
\end{example}

\smallskip

Next we introduce the notion of an invariant class.

\begin{definition}[Invariant class]\label{def_inst}
Given a semialgebraic system $\sys$ with some initial point $\vx_0$ and an invariant cluster $\clus=\{g(\vu,\vx)=0 \mid \vu\in \R^K\setminus\{\vec{0}\} \}$ of $\sys$, where $K\in \N$, an \emphii{invariant class} of $\clus$ at $\vx_0$, denoted by $\class(\clus,\vx_0)$, is the set
 $
 \{g(\vu,\vx)=0 \mid g(\vu,\vx_0)=0$, $\vu\in \R^K \setminus \{\vect{0}\}\}.
 $
\end{definition}

Given an invariant cluster $\clus$, by substituting a specific point $\vx_0$ for $\vx$ in $g(\vu,\vx)=0$, we obtain a constraint $g(\vu,\vx_0)=0$ on $\vu$, which yields a subset $\class(\clus,\vx_0)$ of $\clus$. Apparently, every member of $\class(\clus,\vx_0)$ is an invariant for the trajectory originating from $\vx_0$.

\begin{example}[running example]\label{exam2}
For the given invariant cluster $C^*$ in Example~\ref{exam1} and a given initial point $\vx_0=(4,2)$, we get the invariant class $\class(C^*,\vx_0)=\{u_1 - u_3(x^2 - y^2)=0\mid u_1-12u_3=0, (u_1,u_3)\in \R^2\setminus \{\vect{0}\}\}$. Every member of $\class(C^*,\vx_0)$ is an invariant for the trajectory of $\sys_1$ originating from $\vx_0$. The algebraic variety defined by $\class(C^*,\vx_0)$ is shown in Figure~\ref{fig_example7}.
\end{example}

\begin{figure}[!t]
  \begin{center}
  \includegraphics[height=4cm,keepaspectratio]{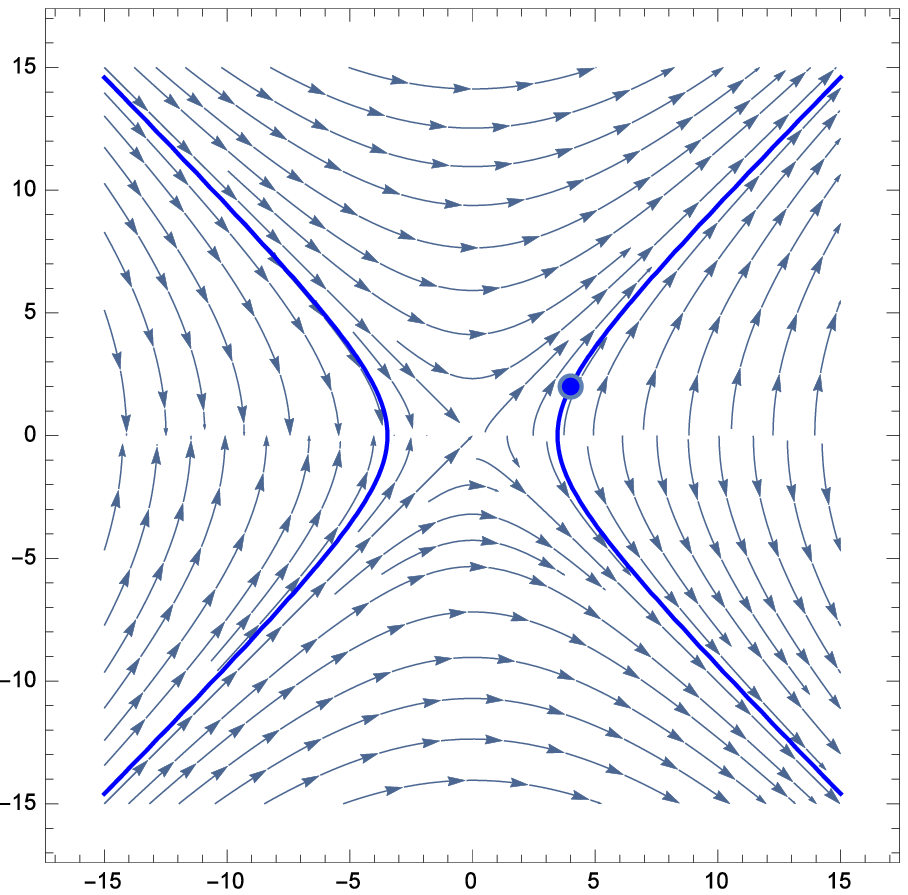}
  \end{center}
  \caption{Example~\ref{exam2}. Curve defined by invariant cluster of $\class(C^*,\vx_0)$ for $\vx_0=(4,2)$.}
  \label{fig_example7}
  \begin{center}
  \includegraphics[height=4cm,keepaspectratio]{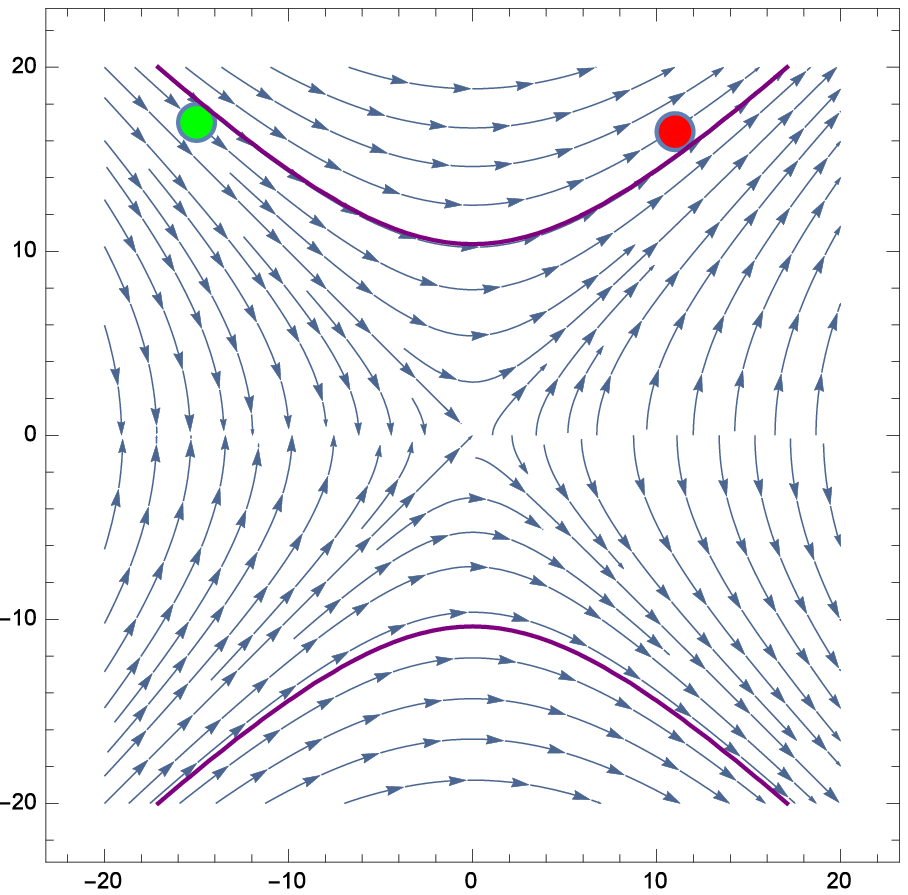}
  \end{center}
  \caption{Example~\ref{example9}. $\Init: (x+15)^2+(y-17)^2 \leq 1$, $\Uns:(x_1-11)^2+(y_1-16.5)^2 \leq 1$.
  }
  \label{fig_example9}
\end{figure}

An invariant class has the following important properties.

\begin{theorem}\label{thrm7}
Given an $n$-dimensional semialgebraic system $\sys$ and an invariant class $D=\{g(\vu,\vx)=0 \mid g(\vu,\vx_0)=0$, $\vu \in \R^K\setminus \{\vect{0}\}\}$ of $\sys$ at a specified point $\vx_0$, let $D_g$ be the set of all invariant polynomials occurring in $D$ and $\pi_{\vx_0}$ be the trajectory of $\sys$ originating from $\vx_0$. Then,
\begin{enumerate}
  \item\label{item1} $\pi_{\vx_0} \subseteq \V(D_g)$;
  \item\label{item2} there must exist a subset $B$ of $D_g$ consisting of $m$ members such that $\ideal{D_g} = \ideal{B}$, where $m$ is the dimension of the hyperplane $g(\vu,\vx_0)=0$ in terms of $\vu$ in $\R^K$ and $\ideal{D_g}$ and $\ideal{B}$ denote the ideal generated by the members of $D_g$ and $B$, respectively.
\end{enumerate}
\end{theorem}
\begin{proof}
  See Appendix~\ref{thrm7proof}.
\end{proof}

\begin{remark}
The first property in Theorem~\ref{thrm7} reveals that a trajectory $\pi_{\vx_0}$ is always contained in the intersection of all the invariants in the invariant class $D$ of $\vx_0$. The second property concludes that the invariant class can be generated by a finite subset $B$ of $D$ if it consists of an infinite number of invariants. The algebraic variety $\V(B)$ (which is equivalent to $\V(D_g)$) forms an overapproximation for $\pi_{\vx_0}$ and the quality of the overapproximation depends largely on the dimension $m$ of $\V(B)$. More specifically, the lower the dimension is, the better the overapproximation is. The ideal case is that, when $m=1$, $\V(B)$ shrinks to an algebraic curve and hence some part of the algebraic curve matches the trajectory precisely. In the case of $m>1$, $\V(B)$ is usually a hypersurface. To make the overapproximation less conservative, we may take the union of multiple invariant classes from different invariant clusters to reduce the dimension of the algebraic variety.
\end{remark}

\subsection{Computing invariant clusters}\label{sec_compute}
In the following, we propose a computation method for invariant clusters. The main idea arises from the following proposition.

\begin{proposition}\label{prop8}
Let $\sys=\consys$ be a semialgebraic system and suppose for a given multivariate polynomial $g\in \R[\vx]$ that the polynomial $\Lfg $ can be written as $\Lfg  = p g + r$. Then $\Lfg  \in \langle g \rangle$ if $r \equiv 0$.
\end{proposition}

According to Proposition~\ref{prop8}, if we can find a polynomial $g(\vx)$ such that the remainder of its Lie derivative w.r.t.\ $g(\vx)$ is identical to $0$, then $g(\vx)=0$ is an invariant of $\sys$. The idea is as follows. We first establish a template $g(\vu,\vx)$ for $g(\vx)$ with parameter $\vu$ and then compute the remainder $r(\vu,\vx)$ of $\Lfg $ w.r.t.\ $g(\vu,\vx)$. According to the computing procedure of a remainder~\cite{cox2010ideals}, $r(\vu,\vx)$ must be of the form $\sum_{i=1}^K \frac{b_i(\vu)}{u_j^d}X^i$, where $d\in\N$, $b_i(\vu)$ are homogeneous polynomials of degree $d+1$ over $\vect{u}$, $u_j$ is the coefficient of the leading term of $g(\vu,\vx)$ by some specified monomial order of $\vect{x}$, and $X^i$ are monomials on $\vx$. Since $r(\vu,\vx) \equiv 0$ implies $u_j\neq 0$ and $b_i(\vu)=0$ for all $i=1,\dots, K$, we obtain a system $\CST$ of homogeneous polynomial equations on $\vu$ plus $u_j\neq 0$ from the coefficients of $r(\vu,\vx)$. Solving $\CST$ can provide a set of invariant clusters of $\sys$ if it exists. Note that all the aforementioned steps can be performed automatically in a mathematical software such as \emph{Maple}. Pseudocode for computing invariant clusters is shown in Algorithm~\ref{algo_invariantdiscovery}. The principle for the loop in line $5$ is that the remainder may vary from the monomial order of $\vx$ and hence produce different solutions. Using multiple orders helps to get more solutions.

\begin{algorithm}[t]
\SetKwData{Coeffs}{Coeffs}
\SetKwData{CFamily}{CFamily}
\SetKwData{Solution}{Solution}
\SetKwFunction{GenPoly}{GenPoly}
\SetKwFunction{NormalForm}{NormalForm}
\SetKwFunction{CollectCoeff}{CollectCoeff}
\SetKwFunction{ContainExistingInvariantFactor}{CollectCoefficientOfX}
\SetKwFunction{RemoveExistingFactorFrmP}{RemoveExistingFactorFrmP}
\SetKwFunction{Solve}{Solve}
\SetKwInOut{Input}{input}
\SetKwInOut{Output}{output~}
\caption{Computation of invariant clusters}\label{algo_invariantdiscovery}
\Input{$f$: $n$-dimensional polynomial vector field; \\$N$: upper bound for the degree of invariant polynomials}
\Output{\CFamily: a set of invariant clusters}
\BlankLine
\CFamily $\leftarrow \emptyset$\;
\For{$i\leftarrow 1$ \KwTo $N$}{
$g_{\vu,\vx}\leftarrow $ generate parameterized polynomial over $\vx$ of degree $i$\;
$L_fg \leftarrow$ compute the Lie derivative of $g_{\vu,\vx}$\;
\ForEach{monomial order $\mathcal{O}$ of $\vx$}
{
$R_{\vu,\vx} \leftarrow$ compute remainder of $L_fg$ w.r.t.\ $g$ by $\mathcal{O}$\;
$\Coeffs \leftarrow$ collect coefficients of $\vx$ in $R_{\vu, \vx}$\;
$\Solution \leftarrow$ solve system $\Coeffs$ on $\vu$\;
  \CFamily $\leftarrow$ \CFamily $\cup\ \Solution$\;
}
}
\end{algorithm}

\begin{remark}[\textbf{Complexity}]
In Algorithm~\ref{algo_invariantdiscovery}, the key steps are computing the remainder in line $4$ and solving the system of equations on $\vu$ in line $8$. The former takes only linear time and hence is very efficient. The latter involves solving a system of homogeneous polynomial equations, which is known to be NP-complete~\cite{ayad2010survey}. In our implementation in \emph{Maple}, we use the command \texttt{solve} for solving the system of equations. In our experiments on nonlinear (parametric) systems of dimension ranging from $2$ to $8$, the solver can return quickly in most cases whether they have solutions or not. Among the solutions we obtained, we found some complex solutions, however, we have not seen nonlinear solutions.
\end{remark}

\begin{example}[running example]\label{example7}
According to Algorithm~\ref{algo_invariantdiscovery}, the steps for computing the invariant clusters of degree~$2$ are as follows:
\begin{enumerate}
  \item Generate the template polynomial of degree $2$:
  \begin{equation*}
   g_2(\vu,\vx) = u_6  x^2 + u_5 xy + u_4 x + u_3 y^2 + u_2 y  + u_1
  \end{equation*}

  \item Compute the Lie derivative $\Lf g_2$ using Definition~\ref{def_Lie}:
  \begin{align*}
  	\Lf g_2 =\frac{\partial g_2}{\partial x}\dot{x}+\frac{\partial g_2}{\partial y}\dot{y} = {u}_{5}{x}^2y+ \left( 2\,{   u}_{3}+2\,{u}_{6}\right) x{y}^2 \\
  	+\ {u}_{2}xy+{y}^3{u}_{5}+{u}_{4}{y}^2
  \end{align*}

  \item Compute the remainder of $\Lf g_2$ w.r.t.\ $g_2$ by graded reverse lexicographic (\emph{grevlex}) order of $(x,y)$. Using this order, the leading term of $\Lf g_2$ and $g_2$ is $u_5x^2y$ and $u_6x^2$, respectively. Then:
  \begin{equation*}
  \begin{split}
  r(\vu,\vx)=&\ \Lf g_2 - \frac{u_5y}{u_6}g_2 =
      {\frac { \left( 2\,{u}_{3}{u}_{6}-u_5^{2}+2\, u_6^2 \right) x{y}^2}{{u}_{6}}} \\
        &+\ {\frac{\left( {u}_{2}{u}_{6}-{u}_{4}{ u}_{5}\right) xy}{{u}_{6}}}+{\frac { \left( -{u}_{3}{u}_{5}+{u}_{5}{u}_{6} \right) {y}^3}{{u}_{6}}}\\
        &+\ {\frac { \left( -{u}_{2}{u}_{5}+{u}_{4}{u}_{6} \right) {y}^2}{{u}_{6}}}-{\frac {{u}_{1}{u}_{5}y}{{u}_{6}}}
  \end{split}
  \end{equation*}

  \item Collect the coefficients of $r(\vu,\vx)$:
\begin{align*}
S:=\left\{{\frac {{u}_{2}{u}_{6}-{u}_{4}{u}_{5}}{{u}_{6}}}, {\frac {-{u}_{3}{u}_{5}+{u}_{5}{u}_{6}}{{u}_{6}}}, {\frac{-{u}_{2}{u}_{5}+{u}_{4}{u}_{6}}{{u}_{6}}}, \right. \\
\left. {\frac {2\,{u}_{3}{u}_{6}-{u}_{5}^2+2\,{u}_{6}^2}{{u}_{6}}}, -{\frac{{u}_{1}{u}_{5}}{{u}_{6}}}\right\}
\end{align*}

  \item Solve the system formed by $S$. To save space, we just present one of the six solutions we obtained:
\begin{equation*}
   C_6=\{u_6  = -u_3 , u_2 = u_4 = u_5  = 0, u_3  = u_3 , u_1  = u_1 \}
\end{equation*}

  \item Substitute the above solution $C_6$ for $\vu$ in $g_2(\vu,\vx)$. We get the following parameterized invariant polynomial:
\begin{equation*}
   g_2(\vu,\vx)=-u_3 x^2 + u_3 y^2 + u_1
\end{equation*}
\end{enumerate}
Note that the other five solutions obtained in step~$5$) are in fact the products of the invariant polynomials $\{u_2y, u_1(x+y), u_1(x-y)\}$, which have been obtained when initially computing the invariants of degree $1$. Hence they cannot increase the expressive power of the set of invariant clusters and should be dropped. The solution presented above is the one we have given in Example~\ref{exam1}.
\end{example}

\subsection{Overapproximating trajectories by invariant classes}
In this section, we address how to overapproximate trajectories precisely by using invariant classes.

Invariant clusters can be divided into two categories according to the number of invariant classes that they can yield by varying the parameter $\vu$.
\begin{enuminNoAnd}{.}
 \item \emphiii{finite invariant cluster}: This kind of invariant cluster can yield only one invariant class no matter how $\vu$ changes. For example, $\{u_1(x-y)=0 \mid u_1\in\R\setminus \{0\}\}$ is such an invariant cluster for the running example. In this case, the trajectories that can be covered by the invariant class is very limited. Moreover, the overapproximation is also conservative due to the high dimension of the algebraic variety defined by the invariant class

 \item \emphiii{infinite invariant cluster}: One such invariant cluster $\clus$ can yield an infinite number of invariant classes $\class(\clus,\vx_0)$ as the initial point $\vx_0$ varies, e.g., the invariant cluster $C^*$ in Example~\ref{exam1}. For the trajectory $\pi_{\vx_0}$, the overapproximating precision of $\class(\clus,\vx_0)$ depends largely on the dimension $m$ of the algebraic variety defined by $\class(\clus,\vx_0)$. The best case is $m=1$ and then it gives a curve-to-curve match in part for the trajectory
\end{enuminNoAnd}

\smallskip

Now, we introduce how to identify the invariant classes for a given point $\vx_0$ from a set of invariant clusters and how to get a finite representation for it. To be intuitive, we first present a $3$-dimensional system and a set of invariant clusters for it.

\begin{example}\label{example8}
  Consider the following semialgebraic system $\sys_2$:
  $
   \left[ \dot{x}, \dot{y}, \dot{z} \right] = \left[ yz, xz, xy \right]
  $.
We obtain a set of invariant clusters consisting of $7$ elements. Here we only present the infinite invariant cluster (for other clusters see Appendix~\ref{append_exam2}).
\begin{align*}
\clus_7=\{g_7(\vu,\vx)&=(-u_5-u_6)x^2+u_5y^2+u_6z^2+u_0 \\
& =0\mid \vu \in \R^3 \setminus \{\vect{0}\}\}
\end{align*}
\end{example}
The invariant clusters are capable of overapproximating all the trajectories of the system $\sys_2$. For any given initial state, how can we identify the invariant classes from the set of invariant clusters? Suppose we want to find the invariant classes that can overapproximate the trajectory from the state $\vx_0=(1,2,3)$. According to Theorem~\ref{thrm7}, we have Algorithm~\ref{algo_invclass} for the purpose.

\begin{algorithm}[t]
\SetKwData{Class}{Class}
\SetKwData{D}{D}
\SetKwData{Basis}{Basis}
\SetKwData{CFamily}{CFamily}
\SetKwData{Solution}{Solution}
\SetKwData{ICls}{ICls}
\SetKwFunction{Solve}{Solve}
\SetKwInOut{Input}{input}
\SetKwInOut{Output}{output~}
\caption{Computation of invariant classes}\label{algo_invclass}
\Input{\CFamily: set of invariant clusters;\\ $\vx_0$: an initial point}
\Output{\ICls: list of invariant classes}
\BlankLine

\ICls $\leftarrow \emptyset$\;
\ForEach{$\clus \in \CFamily$}
{
     $\D \leftarrow \class(\clus,\vx_0)$\;
     \If{$\D \neq \emptyset$}
     {
        $m\leftarrow$ the dimension of the hyperplane $g(\vu,\vx_0)=0$ defining $\D$\;
        \If{$m\geq 1$ }
        {
            $\Basis \leftarrow$ basis $\{u_1,\dots,u_{m}\}$ of the hyperplane $g(\vu,\vx_0)=0$\;
            $\D \leftarrow$ the polynomials $\{g(\vu_1,\vx),\dots,g(\vu_m,\vx)\}, u_i\in \Basis$\;
        }
        $\ICls \leftarrow \ICls \cup \D$\;
     }
}
\end{algorithm}
\begin{remark}
In Algorithm~\ref{algo_invclass}, we enumerate the invariant clusters to find out which one can provide a non-empty invariant class $\class(\clus,\vx_0)$ for $\vx_0$. For a $\class(\clus,\vx_0)$ to be nonempty, the corresponding hyperplane $g(\vu,\vx_0)=0$ must have at least one solution to $\vu\in\R^K\setminus\{\vect{0}\}$, which is equivalent to that its dimension must be at least~$1$. For a hyperplane in $\R^K$, its dimension is equal to $K-1$. Therefore, $\class(\clus,\vx_0)$ must be nonempty if $K>1$ and the basis of the hyperplane can be obtained through a basic linear algebraic computation (which will be illustrated in what follows). However, in the case of $g(\vu,\vx_0)$ being identical to $0$, the hyperplane degenerates to the space $\R^K\setminus \{\vect{0}\}$ and the dimension will be~$K$. Therefore, an invariant class with $K=1$ is nonempty iff $g(\vu,\vx_0)$ is identical to $0$. For example, given an invariant cluster $C^0=\{u_1(x-y)=0 \mid u_1\in\R\setminus \{0\}\}$ and a point $\vect{x}_0=(x_0,y_0)$, $\class(C^0,\vect{x}_0)$ is empty if $x_0\neq y_0$, however, $\class(C^0,\vect{x}_0)$ is equal to $C^0$ if $x_0=y_0$.
\end{remark}

\begin{example}
We continue from Example~\ref{example8}. For the given point $\vx_0=(1,2,3)$, according to Algorithm~\ref{algo_invclass}, we find that only $\class(C_7,\vx_0)=\{g_7(\vu,x,y,z)=0\mid 3u_5 + 8u_6 + u_0 = 0, \vu\in \R^3\setminus\{\vect{0}\}\}$ is nonempty. The dimension of the hyperplane $H:3u_5 + 8u_6 + u_0 = 0$ is~$2$. Since $u_0=-3u_5 -8u_6$, to get the basis of $H$, we can write
$
(u_0,u_5,u_6)=(-3u_5 -8u_6,u_5,u_6) = u_5(-3,1,0) + u_6(-8,0,1).
$
 Hence, we have the following basis for $H$: $\{(-3,1,0), (-8,0,1)\}$. As a result, we get the finite representation $B=\{y^2-x^2-3=0$, $z^2-x^2-8=0\}$ for $\class(C_7,\vx_0)$. It is easy to check by the \emph{Maple} function \texttt{HilbertDimension} that $dim(B)=1$. Therefore, we finally get an algebraic variety $\V(B)$ which provides in part a curve-to-curve match to the trajectory $\pi_{\vx_0}$. The $3$-D vector field and the algebraic curve $\V(B)$ is shown in Figure~\ref{fig:example81} and Figure~\ref{fig:example82}, respectively.
\end{example}

\begin{figure}[t!]
\begin{subfigure}[b]{0.495\linewidth}
  \centering
  \includegraphics[height=40mm,keepaspectratio]{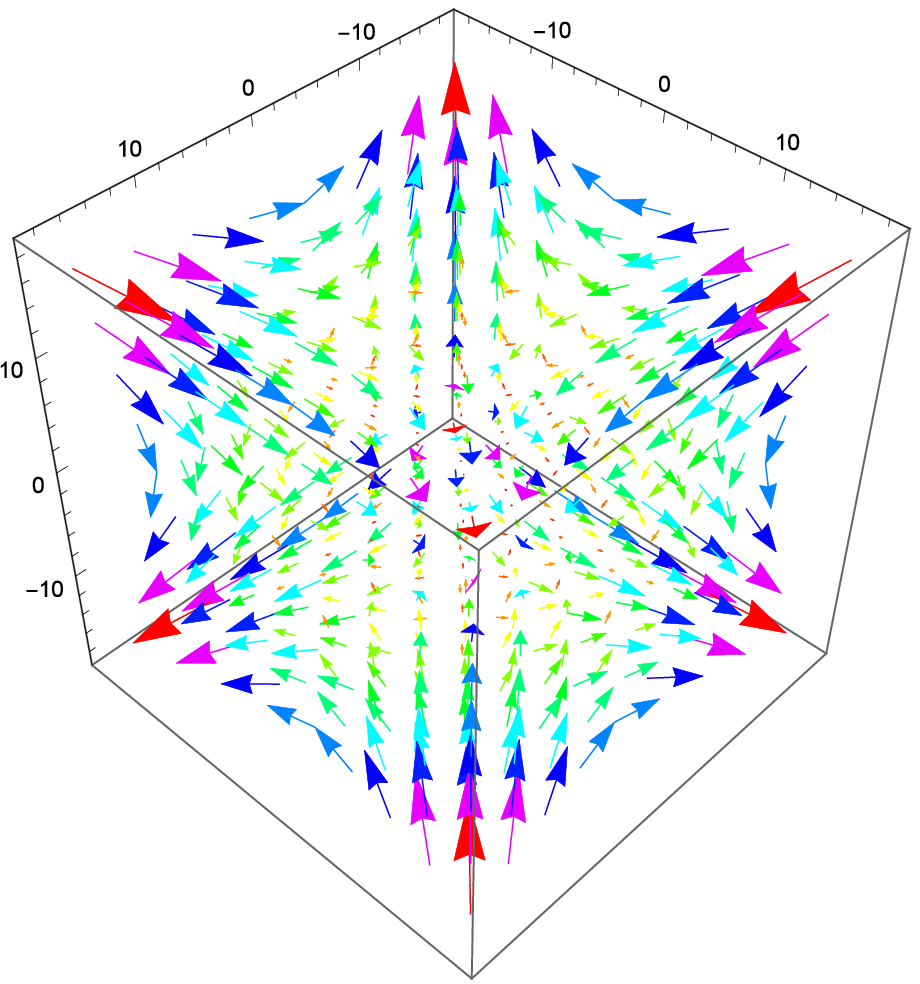}
  \caption{}
  \label{fig:example81}
\end{subfigure}
\begin{subfigure}[b]{0.495\linewidth}
  \centering
  \includegraphics[height=40mm,keepaspectratio]{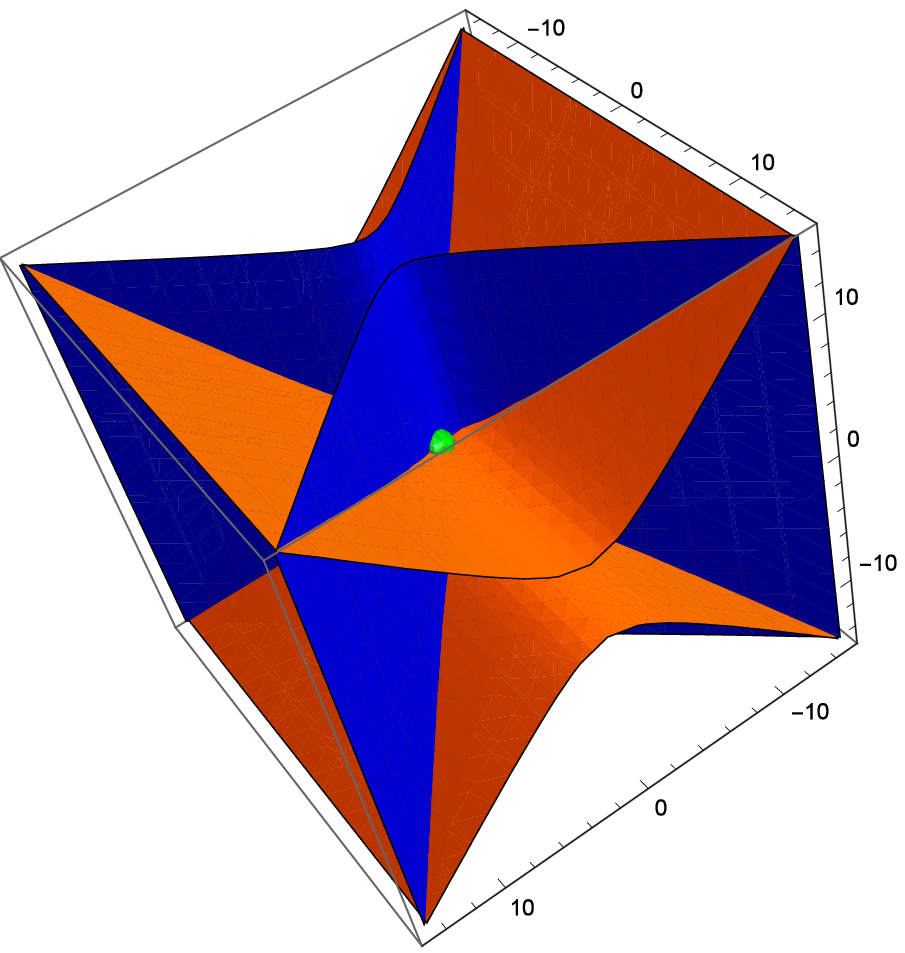}
  \caption{}
  \label{fig:example82}
\end{subfigure}
\caption{(a) 3D vector field of Example~\ref{example8}.
(b) The intersection of the invariants $y^2-x^2-3=0$ (blue) and $z^2-x^2-8=0$ (orange) overapproximates the trajectory originating from $\vx_0=(1,2,3)$ (green ball).}
\end{figure}

%
%

\section{Safety Verification Based on Invariant Clusters}\label{sec:verification}

\subsection{Safety Verification of Continuous Systems}
In this section, we show how to verify a safety property for a nonlinear system based on invariant clusters.

In Section~\ref{sec:computation}, we have demonstrated that a trajectory can be overapproximated by an invariant class.
Since an invariant class is determined uniquely by a single hyperplane $g(\vect{u},\vect{x}_0)=0$ in $\R^K$ for an initial point $\vect{x}_0$, and a hyperplane without constant term (this is the case for $g(\vect{u},\vect{x}_0)=0$) is uniquely determined by its normal vector, we can verify if two states lie on the same trajectory based on the following theorem.

\begin{theorem}\label{thrm8}
Given a semialgebraic system $\sys=\consys$ and an invariant cluster $\clus=\{g(\vect{u},\vect{x})=0 \mid \vect{u}=(u_1,\cdots,u_K)\in \R^K\setminus \{\vect{0}\}\}$ of $\sys$ with $K>1$, where $g(\vect{u},\vect{x})= \sum_{i=1}^{K} \psi_i(\vect{x})u_i$ and $\psi_i(\vect{x})\in \R[\vect{x}]$, let \Init be the initial set and \Uns be the unsafe set. Then, if there exists a pair of states $(\vect{x}_1,\vect{x}_2)\in \Init\times\Uns$ such that $\vect{x}_1$ and $\vect{x}_2$ lie on the same trajectory, one of the following two formulae must hold:
\begin{enumerate}
\item[(i)] \parbox{\linewidth}{%
\begin{equation}
\label{normalvec}
\begin{gathered}
\exists k\in \R\backslash\{0\}:\ k \psi_i(\vect{x}_1) = \psi_i(\vect{x}_2), i = 1,\dots,K
\end{gathered}
\end{equation}
}
\item[(ii)] \parbox{\linewidth}{%
\begin{equation}\label{zeronormalvec}
\psi_i(\vect{x}_1) = \psi_i(\vect{x}_2) = 0, i=1,\dots,K
\end{equation}
}
\end{enumerate}

Moreover, if some $\psi_i(\vect{x})\equiv 1$, i.e. $g(\vect{u},\vect{x})$ contains a constant term $u_i$, then formula~\eqref{normalvec} simplifies to
\begin{equation}\label{normalvec1}
\psi_i(\vect{x}_1)) = \psi_i(\vect{x}_2), i = 1,\dots,K
\end{equation}
\end{theorem}
\begin{proof}
  See Appendix~\ref{proofofthrm8}.
\end{proof}

\begin{remark}
Instead of computing the invariants explicitly, Theorem~\ref{thrm8} provides an alternative way to verify whether two states $\vect{x}_1,\vect{x}_2$ lie on the same trajectory by checking that the difference between the normal vectors of $g(\vect{u},\vect{x}_1)=0$ and $g(\vect{u},\vect{x}_2)=0$. Let us take Example~\ref{example8} for illustration. We think of $\clus_7$ as a hyperplane over $\vu\in\R^3$: $u_0+(y^2-x^2)u_5 + (z^2-x^2)u_6=0$, hence the corresponding normal vector is $\vect{\mathcal{N}}(\vx)=(1,y^2-x^2,z^2-x^2)$. Given two random points $\vect{x}_1=(1,2,3)$ and $\vect{x}_2=(5,\sqrt{27},\sqrt{34})$, it is easy to verify that $\vect{\mathcal{N}}(\vect{x}_1)\neq \vect{\mathcal{N}}(\vect{x}_2)$, which means that $\vect{x}_1$ and $\vect{x}_2$ are not on the same trajectory. This can be verified in another way, as we know that the invariant class of $\vect{x}_1$ is $\{y^2-x^2-3=0, z^2-x^2-8=0\}$ and $\vect{x}_2$ does not belong to its solution set.
\end{remark}

Now we demonstrate how to verify a safety property of semialgebraic systems. Assume \Init and \Uns can be written as semialgebraic sets, i.e., $\Init=\{\vx_1\in \R^n \mid p_{i_1}(\vx_1)=0,q_{j_1}(\vx_1)\geq 0,r_{k_1}(\vx_1)>0, i_1=1,\dots, l_1$, $j_1=1\dots m_1$, $k_1=1,\dots, n_1\}$ and $\Uns=\{\vx_2 \in \R^n \mid p_{i_2}(\vx_2)=0,q_{j_2}(\vx_2)\geq 0,$ $r_{k_2}(\vx_2)>0$, $i_2=l_1+1,\dots, l_1+l_2$, $j_2=m_1+1,\dots, m_1+m_2$, $k_2=n_1+1, \dots, n_1+n_2\}$. Then we have the following theorem for deciding the safety of a semialgebraic system.
\begin{theorem}\label{thrm2}
Given a semialgebraic system $\sys=\consys$ and invariant cluster $\clus=\{g(\vect{u},\vect{x})=0 \mid \vect{u}\in \R^K\setminus\{\vect{0}\}\}$ of $\sys$ with $K\geq 2$. Suppose the normal vector of the hyperplane $g(\vect{u},\vect{x})=0$ over $\vu$ is $(1,\psi_1(\vect{x}), \dots ,\psi_K(\vect{x}))$. Then the system $\sys$ is safe if there exists the following polynomial identity
\vspace*{-1mm}
\begin{align}\label{eq12}
  &\sum_{k=1}^K \gamma_k\cdot (\psi_k(\vect{x}_1)-\psi_k(\vect{x}_2)) + \sum_{i=1}^{l_1+l_2} \beta_i p_i \nonumber \\
  &+\hspace*{-1mm} \sum_{v\in \{0,1\}^{m_1+m_2}}(\sum_t b_{tv})^2\cdot \hspace*{-1mm} \prod_{j=1}^{m_1+m_2} q_j^{v_j}\\
  &+\hspace*{-1mm} \sum_{v\in \{0,1\}^{n_1+n_2}}(\sum_t c_{tv})^2\cdot \hspace*{-1mm} \prod_{k=1}^{n_1+n_2} r_k^{v_k} + \sum_w s_w^2 + \hspace*{-1mm} \prod_{k=1}^{n_1+n_2}r_k^{d_k} = 0 \nonumber
\end{align}
\vspace*{-1mm}
where $d_k\in \N$ and $\beta_j, \gamma_k, b_{tv}, c_{tv}, s_w$ are polynomials in $\R[\vx_1,\vx_2]$.
\end{theorem}
\begin{proof}
  See Appendix~\ref{proofofthrm2}.
\end{proof}

\begin{remark}
Theorem~\ref{thrm2} transforms the safety verification problem into a decision problem about the existence of a real solution of a system of polynomial equations and inequalities. As noted in Remark~\ref{remark_sos}, this decision problem can be solved by \SOS programming. Our implementation uses the efficient tool \emph{SOSTOOLS}~\cite{prajna2005sostools}. Appendix~\ref{sosprogramming} contains more information on \SOS programming.
\end{remark}

In Theorem~\ref{thrm2}, we deal with a general semialgebraic system where the initial set and the unsafe set are represented by a set of polynomial equations and inequalities. However, if the system is described by much simpler set representations such as a single polynomial equation or inequality, the programming problem can be simplified correspondingly. For example, if both sets can be represented or overapproximated by a single polynomial equation $I(\vx_1)=0$ and $U(\vx_2)=0$, respectively, then the programming problem is simplified to (see \cite{stengle1974nullstellensatz})
\vspace*{-3mm}
\begin{equation}\label{sosform}
  \sum_{j=1}^K\alpha_j(\psi_j(\vx_1)-\psi_j(\vx_2)) + \beta I + \theta U -1\quad \textrm{is an \SOS}
\vspace*{-2mm}
\end{equation}
where $(\psi_1(\vx), \dots ,\psi_K(\vx))$ is the same as in Theorem~\ref{thrm2} and $\alpha_j, \beta, \theta\in \R[\vx_1,\vx_2]$.
For how to overapproximate a compact set we refer the reader to Appendix~\ref{overapprox}. The algorithm for safety verification
based on the condition~\eqref{sosform} is shown in Algorithm~\ref{algo_verification}.

\begin{algorithm}[t!]
\SetKwData{IsSafe}{IsSafe}
\SetKwData{True}{True}
\SetKwData{False}{False}
\SetKwData{Solution}{Solution}
\SetKwInOut{Input}{input}
\SetKwInOut{Output}{output~}
\caption{Safety verification}\label{algo_verification}
\Input{$\vect{\psi}$: the $K$-dimensional normal vector of an invariant cluster;\\I($\vx_1$): the initial set;\\ U($\vx_2$): the unsafe set;\\$N$: the maximum degree of programming polynomials $\vect{\alpha}, \beta, \theta$}
\Output{\IsSafe: whether the system is safe}
\BlankLine
$\IsSafe\leftarrow \False$\;
\For{$i \leftarrow 1$ \KwTo $N$}{
$\vect{\alpha} \leftarrow $ generate a vector of polynomials of degree $i$ for $\psi$\;
$\beta \leftarrow $  generate a polynomial of degree $i$ for I($\vx_1$)\;
$\theta \leftarrow $ generate a polynomial of degree $i$ for U($\vx_2$)\;
$P \leftarrow \sum_{j=1}^K \alpha_j (\vect{\psi}_j(\vx_1)-\vect{\psi}_j(\vx_2)) +\beta$I$ + \theta$U $- 1$\;
$\Solution \leftarrow$ perform \SOS programming on $P$\;
\If{ \Solution is found}
{
   $\IsSafe\leftarrow \True$\;
   break\;
}
}
\end{algorithm}

The next example demonstrates the application of the verification method.

\begin{example}[running example 2]\label{example9}
Given the semialgebraic system $\sys_3$ by
$
 \left[ \dot{x}, \dot{y} \right] = \left[ y^2, xy \right]
$
and the initial set be $\Init=\{(x,y)\in \R^2 \mid I(x,y)=(x+15)^2+(y-17)^2 -1\leq 0\}$, verify that if the unsafe set $\Uns=\{(x,y)\in \R^2 \mid U(x,y) = (x-11)^2+(y-16.5)^2-1 \leq 0\}$ can be reached.
The parameter space of $\clus^*=\{g(\vect{u},\vect{x})=u_1 - u_3(x^2 - y^2)=0 \mid (u_1,u_3)\in \R^2\setminus \{\vect{0}\}\}$ has dimension~$1$
and hence can provide an invariant class for every state in \Init and \Uns. The normal vector of the hyperplane $g(\vect{u},\vect{x})=0$ is $(1,\psi_1(x,y))=(1,y^2-x^2)$. Let $\varphi(x_1,y_1,x_2,y_2)=\psi_1(x_1,y_1)-\psi_1(x_2,y_2)$. By Theorem~\ref{thrm2}, to verify the safety property, we only need to verify that the following system of equations has no real solution.

\vspace*{-\baselineskip}
\begin{equation*}
\begin{split}
  &I(x_1,y_1) = (x_1+15)^2+(y_1-17)^2 -1 = 0\\
  &U(x_2,y_2)=(x_2-11)^2+(y_2-16.5)^2-1 = 0\\
  &\varphi(x_1,y_1,x_2,y_2) = y_1^2-x_1^2 -( y_2^2-x_2^2) = 0
\end{split}
\end{equation*}
Note that we substitute $(x_1,y_1)$, $(x_2,y_2)$ for $(x,y)$ in $I(x,y)$ and $U(x,y)$, respectively, to denote the different points in \Init and \Uns. To prove that the system is safe, we need to find $\alpha_i\in \R[x_1,y_1,x_2,y_2], i=1, 2, 3$ such that $\Prog = \alpha_1 I + \alpha_2 U + \alpha_3 \varphi -1$ is a sum-of-squares. Finally, we found three polynomials of degree $2$ for $\alpha_i$, respectively (see Appendix~\ref{verifydata} for the expressions of $\alpha_i$ and $\Prog$), hence the system is safe. As shown in Figure~\ref{fig_example9}, although the relative position of \Uns to the reachable set of \Init is very close, we can still verify the safety property using an invariant cluster. However, we failed to find a barrier certificate for this system by using the methods in~\cite{kong2013exponential,prajna2004safety}.
\end{example}

\smallskip

In Theorem~\ref{thrm2}, we present a sufficient condition for deciding if a semialgebraic system is safe. The theory originates from the fact that the system is safe if there is no invariant class intersecting both the initial and the unsafe set, which is equivalent to that the formula~\eqref{eq12} holds. To verify the latter, we need to find a set of witness polynomials by \SOS programming. However, as the dimension of the system increases, the number of parametric polynomials involved increases correspondingly, which also leads to an increase in computational complexity. In what follows, we present a new method for safety verification, which avoids the aforementioned problem. The new method is based on Proposition~\ref{prop9}, that is, for any polynomial $g(\vect{x})$ satisfying $\mathcal{L}_{\vect{f}} g\in \ideal{g}$, $g(x)\sim 0$ is an invariant for any $\sim\ \in \{<, \leq, =, \geq, >\}$.
\begin{proposition}\label{prop10}
Given a semialgebraic system $\sys=\consys$ and an invariant cluster $\clus=\{g(\vect{u},\vect{x})=0 \mid \vect{u}\in \R^K\setminus \{\vect{0}\}\}$ of $\sys$, let $\Init$ and $\Uns$ be the initial set and the unsafe set, respectively. Then, the system is safe if there exists a $\vect{u}^* \in \R^K\setminus \{\vect{0}\}$ such that
\begin{align}
 &\forall \vect{x} \in \Init: g(\vect{u}^*,\vect{x}) \geq 0 \label{cond1}\\
 &\forall \vect{x} \in \Uns: g(\vect{u}^*,\vect{x}) < 0 \label{cond2}
\end{align}
\end{proposition}
\begin{proof}
  See Appendix~\ref{proofofprop10}.
\end{proof}

According to Proposition~\ref{prop10}, to verify the safety property, it suffices to find a $\vect{u}^*\in \R^K\setminus \{\vz\}$ which satisfies the constraints~\eqref{cond1} and \eqref{cond2}. There are some constraint solving methods available, e.g, \emph{SMT} solvers. However, the high complexity of \emph{SMT} theory limits the applicability of the method. In the following, we transform the above constraint-solving problem into an \SOS programming problem, which can be solved efficiently. We write $\vect{P}(\vx)\succeq \vz$ to denote $p_i(\vx)\geq 0, i=1,\dots, m$ for a polynomial vector $\vect{P}(x)=(p_1(\vx),\dots, p_m(\vx))$.

\begin{proposition}\label{prop11}
Given a semialgebraic system $\sys=\consys$ and an invariant cluster $\clus=\{g(\vect{u},\vect{x})=0 \mid $, $\vect{u}\in\R^K\setminus \{\vz\}\}$ of $\sys$ and a constant $\epsilon \in \R_{>0}$, let $\Init=\{\vect{x}\in\R^n \mid \vect{I} \succeq \vz, \vect{I}\in \R[\vect{x}]^{m_1}\}$ and $\Uns=\{\vect{x}\in \R^n \mid \vect{U} \succeq \vz, \vect{U}\in \R[\vect{x}]^{m_2}\}$. Then, the system is safe if there exist a $\vect{u}^*\in \R^K\setminus\{\vz\}$ and two \SOS polynomial vectors $\vect{\mu}_1\in \R[\vect{x}]^{m_1}$ and $\vect{\mu}_2\in \R[\vect{x}]^{m_2}$ such that the following polynomials are \SOS polynomials.
\begin{align}
  &\quad g(\vect{u}^*,\vect{x}) - \vect{\mu}_1 \cdot \vect{I} \label{cond21}\\
  &-g(\vect{u}^*,\vect{x}) - \vect{\mu}_2\cdot \vect{U} - \epsilon \label{cond22}
\end{align}
\end{proposition}
\begin{proof}
  See Appendix~\ref{proofofprop11}.
\end{proof}

Similar to Theorem~\ref{thrm2}, Proposition~\ref{prop11} also reduces to an \SOS programming problem. However, the ideas behind these two theories are different in that by Theorem~\ref{thrm2} we attempt to prove no invariant class which overapproximates trajectory can intersect both $\Init$ and $\Uns$, while by Proposition~\ref{prop11} we mean to find a hypersurface which is able to separate the reachable set of $\Init$ from the unsafe set $\Uns$. Apparently, there must exist no invariant class intersecting both $\Init$ and $\Uns$ if there exists such a hypersurface, but not vice versa. Hence the latter is more conservative than the former, but it is also more efficient in theory because it usually involves less unknown polynomials. For example, for an $n$-dimensional system with $\Init$ and $\Uns$ defined by a single polynomial inequality, respectively, we usually need $n+1$ unknown polynomials for the former method, however, we need only $2$ for the latter. See Algorithm~\ref{algo_prop11} for the pseudocode of the method based on Proposition~\ref{prop11}.

\begin{algorithm}[t!]
\SetKwData{IsSafe}{IsSafe}
\SetKwData{True}{True}
\SetKwData{False}{False}
\SetKwData{Solution}{Solution}
\SetKwInOut{Input}{input}
\SetKwInOut{Output}{output~}
\caption{Algorithm for safety verification based on \SOS programming}\label{algo_prop11}
\Input{$g(\vu,\vx)$: the invariant polynomial defining an invariant cluster; \\$\vect{I}$: the polynomial vector describing initial set;\\ $\vect{U}$: the polynomial vector describing unsafe set;\\$N$: the maximum degree of the polynomial vectors $\vect{\mu}_1,\vect{\mu}_2$;\\ $\epsilon$: a positive real number}
\Output{\IsSafe: whether the system is safe;\\
        $g(\vu^*,\vx)$, $\vect{\mu}_1$, $\vect{\mu}_2$: the feasible solution if \IsSafe;\\
        }
\BlankLine
$\IsSafe\leftarrow \False$\;
\For{$i \leftarrow 0$ \KwTo $N$}{
$\vect{\mu}_1 \leftarrow $  generate a parametric polynomial vector of degree $2i$ for $\vect{I}$\;
$\vect{\mu}_2 \leftarrow $ generate a parametric polynomial vector of degree $2i$ for $\vect{U}$\;
$p_1 \leftarrow g(\vu,\vx) - \vect{\mu}_1\cdot \vect{I} $\;
$p_2 \leftarrow -g(\vu,\vx) - \vect{\mu}_2\cdot \vect{U} -\epsilon$\;
$\Solution \leftarrow$ perform \SOS programming on $\{\vect{\mu}_1,
\vect{\mu}_2, p_1, p_2\}$\;
\If{ \Solution is found}
{
   $\IsSafe\leftarrow \True$\;
   $[g(\vu^*,\vx),\vect{\mu}_1, \vect{\mu}_2] \leftarrow $ get the feasible solution from \Solution\;
   break\;
}
}
\end{algorithm}
%

Let us use the following example to demonstrate the application of Algorithm~\ref{algo_prop11}.

\begin{example}\label{example10}
 Given the semialgebraic system $S_4$ by
 \begin{gather*}
 \begin{bmatrix}\dot{x} \\ \dot{y}\end{bmatrix} = \begin{bmatrix} y^2-2y \\ x^2+2x\end{bmatrix},
 \end{gather*}
let the initial set be $\Init=\{(x,y)\in \R^2 \mid I(x,y)=1 - (x + 6.0)^2 - (y + 6.0)^2\geq 0\}$, decide whether the unsafe set $\Uns=\{(x,y)\in \R^2 \mid U(x,y) = 1 - (x - 8.2)^2 - (y - 4.0)^2 \geq 0\}$ can be reached.
\end{example}

By Algorithm~\ref{algo_invariantdiscovery}, we first get an invariant cluster $C_4=\{g(\vu,\vx)=\frac{1}{3}u_3x^3-\frac{1}{3}u_3y^3+u_3x^2+u_3y^2+u_1 = 0 \mid \vu\in \R^2\}$ for the system $S_4$, then by Algorithm~\ref{algo_prop11}, we find an invariant $g(\vu^*,\vx) = 0$ with $\vu^*=(-3081.9,7.1798)$ from the invariant cluster $C_4$: $g(\vu^*,\vx)=7.1798x^3 - 7.1798y^3 + 21.539x^2 + 21.539y^2- 3081.9$ and two polynomials $\mu_1(\vx), \mu_2(\vx)$ of degree $2$. The stream plot of $S_4$ and the plot of $g(\vu^*,\vx)=0$ are shown in Figure~\ref{fig_example10}. Note that we failed to find a barrier certificate by using the method in \cite{prajna2004safety} and \cite{kong2013exponential} for this system.

 \begin{figure}[!t]
  \centering
   \includegraphics[scale=0.5]{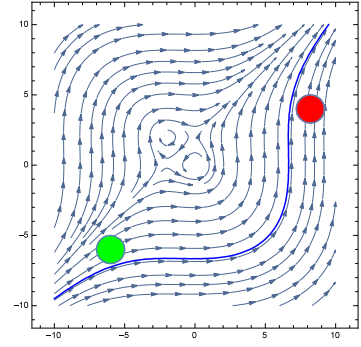}
 \caption{Example~\ref{example10}. Invariant $g(\vu^*,\vx)=0$ (blue curve) separating reachable set of $\Init$ (green patch) from $\Uns$ (red patch).}
 \label{fig_example10}
 \end{figure}

\subsection{Safety Verification of Hybrid Systems}\label{sec:hybridverification}
In this section, we extend the safety verification method for continuous systems based on invariant clusters to semialgebraic hybrid systems.

A hybrid system consists of a set of locations and a set of discrete transitions between locations. In general, different locations have different continuous dynamics and hence correspond to different invariant clusters. An invariant for the hybrid system can be synthesized from the set of invariant clusters of all locations. The idea is to pick out a polynomial $g_l(\vu_l^*,\vx)$ from the respective invariant cluster $C_l$ for each location $l$ such that $g_l(\vu_l^*,\vx) \geq 0$ is an invariant for the location $l$ and all the invariants coupled together by the constraints at the discrete transitions form a hybrid invariant for the hybrid system. The aforementioned idea is formalized in Proposition~\ref{prop12}.

\begin{proposition}\label{prop12}
Given an $n$-dimensional hybrid system $\mathcal{H}=\langle L,X,E,R,G,I,F\rangle$ and a set of invariant clusters $\{C_l, l=1,\dots, n\}$, where $C_l=\{g_l(\vu_l,\vect{x})=0 \mid \vu_l\in \mathbb{R}^{K_l} \setminus \{\vect{0}\}\}$ with $K_l>1$ is an invariant cluster for location $l$, the system $\mathcal{H}$ is safe if there exists a set $S_{\vu}=\{\vu_l^*\in \mathbb{R}^{K_l}\setminus \{\vz\}, l=1,\dots ,n\}$ such that, for all $l\in L$ and $(l,l')\in E$, the following formulae hold:
\begin{align}
 &\forall \vx \in \HInit(l): g_l(\vu_l^*,x) \geq 0\label{cond31}\\
 &\hspace*{-2pt}
 \begin{gathered}
 \forall \vx \in G(l,l'), \forall \vx' \in R((l,l'),\vx): g_l(\vu_l^*,\vx)\geq 0 \\
 \implies g_{l'}(\vu_{l'}^*,\vx')\geq 0 \label{cond32}\\
 \end{gathered} \\
 &\forall \vx \in I(l)\cap \HUns(l): g_l(\vu_l^*,\vx) < 0\label{cond33}
\end{align}
where $\HInit(l)$ and $\HUns(l)$ denote respectively the initial set and the unsafe set at location $l$.
\end{proposition}
\begin{proof}
See Appendix~\ref{proofprop12}.
\end{proof}

Similar to Proposition~\ref{prop10}, we further transform the problem into an \SOS programming problem. Consider a semialgebraic hybrid system $\mathcal{H} = \langle L,X,E,R,G,I,F\rangle$, where the mappings $R,G$, and $I$ are defined in terms of polynomial inequalities as follows:
\begin{itemize}
  \item $G:(l,l')\mapsto\{\vx\in\mathbb{R}^n \mid \vect{G}_{ll'}\succeq~0, \vect{G}_{ll'}\in\mathbb{R}[\vx]^{m_{ll'}}\}$
  \item $R:(l,l',\vx)\mapsto~\{\vx\in\mathbb{R}^n \mid \vect{R}_{ll'\vx}\succeq~0, \vect{R}_{ll'\vx}\in\mathbb{R}[\vx]^{n_{ll'}}\}$
  \item $ I:l\mapsto\{\vx\in\mathbb{R}^n\mid \vect{I}_l\succeq 0, \vect{I}_l\in\mathbb{R}[\vx]^{p_l}\}$
\end{itemize}
and the mappings of the initial and the unsafe set are defined as follows:
\begin{itemize}
  \item $\HInit:l\mapsto\{\vx\in\mathbb{R}^n \mid \vect{\HInit}_l\succeq 0, \vect{\HInit}_l\in\mathbb{R}[\vx]^{r_l}\}$
  \item $\HUns:l\mapsto\{\vx\in\mathbb{R}^n \mid \vect{\HUns}_l\succeq 0, \vect{\HUns}_l\in\mathbb{R}[\vx]^{s_l}\}$
\end{itemize}
where $m_{ll'}$, $n_{ll'}$, $r_l$, $p_l$ and $s_l$ are the dimensions of the polynomial vector spaces.
Then we have the following proposition for safety verification of the semialgebraic hybrid system $\mathcal{H}$.

\begin{proposition}\label{prop13}
Let the hybrid system $\mathcal{H}$, the initial set mapping $\HInit$, and the unsafe set mapping $\HUns$ be defined as above. Given a set of invariant clusters $\{C_l, l=1,\dots, n\}$ of $\mathcal{H}$, where $C_l=\{g_l(\vu_l,\vx)=0 \mid \vu_l\in \mathbb{R}^{K_l}\setminus \{\vect{0}\}\}$ with $K_l>1$ is an invariant cluster for location $l$, a set $S_{\gamma}=\{\gamma_{ll'}\in\mathbb{R}_{\geq 0}, (l,l')\in E\}$ of constants, and a constant vector $\vect{\epsilon} \in \R_{>0}^n$, the system is safe if there exists a set $S_u=\{\vu_l^*\in \mathbb{R}^{K_l}\setminus\{\vz\}, l=1,\dots, n\}$ and five sets of \SOS polynomial vectors $\{\vect{\theta}_l\in\mathbb{R}[\vx]^{s_l}, l\in L\}$, $\{\vect{\kappa}_{ll'} \in \mathbb{R}[\vx]^{p_{ll'}}, (l,l') \in E\}$, $\{\vect{\sigma}_{ll'}\in\mathbb{R}[\vx]^{q_{ll'}}, (l,l')\in E\}$, $\{\vect{\eta}_l\in\mathbb{R}[\vx]^{t_l}, l\in L\}$, and $\{\vect{\nu}_l\in\mathbb{R}[\vx]^{w_l}, l\in L\}$ such that the following polynomials are \SOS for all $l\in L$ and $(l,l')\in E$:
\begin{align}
&g_l(\vu_l^*,\vx) - \vect{\theta}_l\cdot \vect{\HInit}_l\label{contCond1}\\
&g_{l'}(\vu_{l'}^*,\vx') - \gamma_{ll'} g_l(\vu_l^*,\vx) -\vect{\kappa}_{ll'}\cdot G_{ll'}-\vect{\sigma}_{ll'} \cdot R_{ll'\vx} \label{contCond2}\\
&-\vect{\nu}_l\cdot I_l- \vect{\eta}_l\cdot \vect{\HUns}_l - g_l(\vu_l^*,\vx) - \epsilon_l \label{contCond3}
\end{align}
\end{proposition}
\begin{proof}
Similar to the proof of Proposition~\ref{prop11}, we can easily derive the formulae~\eqref{cond31}--\eqref{cond33} from the \SOS's~\eqref{contCond1}--\eqref{contCond3}, respectively. Then, by Proposition~\ref{prop12}, the system is safe.
\end{proof}

The algorithm for computing invariants for semialgebraic hybrid systems based on Proposition~\ref{prop13} is very similar to Algorithm~\ref{algo_prop11} for semialgebraic continuous systems except that it involves more \SOS constraints on continuous and discrete transitions.

\section{Implementation and Experiments}\label{sec:evaluation}

Based on the approach presented in this paper, we implemented a prototype tool in \emph{Maple} and \emph{Matlab}, respectively. In \emph{Maple}, we implemented the tool for computing invariant clusters and identifying invariant classes based on the remainder computation of the Lie derivative of a polynomial w.r.t. its \Groba and solving the system of polynomial equations obtained from the coefficients of the remainder. In \emph{Matlab}, we implemented the tool for safety verification based on the \SOS programming tool package \emph{SOSTOOLS}. Currently, we manually transfer the invariant clusters computed in \emph{Maple} to \emph{Matlab} for safety verification. In the future, we will integrate the two packages into a single tool.

\smallskip

Now, we present the experimental results on nonlinear benchmark systems, run on a laptop with an 3.1GHz \emph{Intel Core i7} CPU and~$8$GB memory. 
\def\mycolA{p{8mm}}
\def\mycolB{p{11mm}}
\def\mycolC{p{8mm}}
\begin{table*}[t]
  \caption{Benchmark results for the Longitudinal Motion of an Airplane (B1), the Looping Particle system (B2), the Coupled Spring-Mass system (B3), and the 3D-Lotka-Volterra System (B4).}
  \centering
  \scalebox{0.9}{
\begin{tabular}{|c|\mycolA|\mycolA|\mycolA|\mycolA|\mycolB|\mycolB|\mycolB|\mycolB|\mycolC|\mycolC|\mycolC|\mycolC|}
 \hline
 Degree of & \multicolumn{4}{c|}{No. of variables} & \multicolumn{4}{c|}{Running time (sec)} & \multicolumn{4}{c|}{No. of invariant clusters} \\
 \cline{2-13}
 invariants & \centering B1 & \centering B2 & \centering B3 & \centering B4 & \centering B1 & \centering B2 & \centering B3  & \centering B4 & \centering B1 & \centering B2 & \centering B3  & \centering B4\tabularnewline
 \hline
 \hline
 1  & \centering 9 & \centering 4 & \centering 6  & \centering 4 & \centering 0.016 & \centering 0.015& \centering 0.047 & \centering <0.001 & \centering 0 & \centering 0 & \centering 0 & \centering 2 \tabularnewline
 2  & \centering 45 & \centering 10 & \centering 21  & \centering 10 &\centering 0.031 & \centering 0.047 & \centering 0.078 & \centering 0.031 & \centering 1 & \centering 1 & \centering 0& \centering 3 \tabularnewline
 3  & \centering 165 & \centering 20 & \centering 56  & \centering 20 &  \centering 0.484 & \centering 0.049 & \centering 0.250  & \centering 0.109 & \centering 1 & \centering 0 & \centering 1  & \centering 7 \tabularnewline
 4  & \centering 495 & \centering 35 & \centering 126 & \centering 35 & \centering 3.844 & \centering 0.156& \centering 1.109 & \centering 0.312 & \centering 1 & \centering 1 & \centering 1 & \centering 6\tabularnewline
 5  & \centering 1287 & \centering 56 & \centering 252 & \centering 56 & \centering 25.172 & \centering 0.703 & \centering 6.641 & \centering 0.750 & \centering 1 & \centering 0 & \centering 1 & \centering 6\tabularnewline
 6  & \centering 3003 & \centering 84 & \centering 462 & \centering 84 & \centering 200.903 & \centering 3.000 & \centering 32.109 & \centering 1.641 & \centering 1 & \centering 1 & \centering 1  & \centering 16\tabularnewline
 \hline
\end{tabular}
}
\label{tbl:benchmarks}
\end{table*}

\subsection{Longitudinal Motion of an Airplane}\label{motionairplane}
In this experiment, we study the $6$th order longitudinal equations of motion that captures the vertical motion (climbing,
descending) of an airplane (\cite{stengel2004flight}, Chapter $5$). Let $g$ denote the gravity acceleration, $m$ the total mass of an airplane,
$M$ the aerodynamic and thrust moment w.r.t. the $y$ axis, $(X,Z)$ the aerodynamics and thrust forces w.r.t. axis $x$ and $z$, and $I_{yy}$ the second diagonal element of its inertia matrix. Then the motion of the airplane is described as follows.
\begin{align*}
\dot{v}&=\frac{X}{m} - g \sin (\theta) - q w,  &\dot{w}&=\frac{Z}{m} + g \cos (\theta) + q v,\\
\dot{x}&=w \sin (\theta) + v \cos (\theta), &\dot{z}&=-v \sin (\theta) + w  \cos (\theta),\\
\dot{\theta}&=q, &\dot{q}&=\frac{M}{I_{yy}},
\end{align*}

where the meanings of the variables are as follows: $v$: axial velocity, $w$: vertical velocity, $x$: range, $z$: altitude, $q$: pitch rate, $\theta$: pitch angle.

To transform the above system into a semialgebraic system, we first introduce two additional variables $d_1$, $d_2$ such that $d_1 = \sin(\theta)$, $d_2=\cos(\theta)$ and then substitute $d_1$ and $d_2$ respectively for $\sin(\theta)$ and $\cos(\theta)$ in the model. In addition, we get two more constraints $\dot{d_1}=q d_2$ and $\dot{d_2}=-q d_1$. As a result, the dimension of the system rises to $8$. For this system, using the method in \cite{ghorbal2014characterizing}, Ghorbal et al. spent $\mathbf{1}$ hour finding three invariant polynomials of degree $3$ on a laptop with a 1.7GHz \emph{Intel Core i5} CPU and $4$GB memory. Using our method, we spent only $\mathbf{0.406}$ seconds obtaining an invariant cluster $g_9(\vu,\vx)=0$ of degree $3$, which is presented in Appendix~\ref{invairplane}. By the constraint $d_1^2+d_2^2 = 1$, we reduce the normal vector of the hyperplane $g_9(\vu,\vx)=0$ in $\vu$ to $(1,\psi_1,\psi_2,\psi_3)$, where $\psi_1,\psi_2,\psi_3$ are defined as follows.

\vspace*{-\baselineskip}
\begin{equation*}
  \begin{split}
\psi_1={}&{\frac {Mmz}{{I_{yy}}\,Z}}+{\frac {gm\theta}{Z}}+ \left( {\frac {mqv}{Z}}+1 \right) \sin \left( \theta \right)\\
       & + \left({\frac {X}{Z}} -{\frac {mqw}{Z}} \right) \cos \left( \theta \right)  \\
\psi_2={}&-{\frac {Xz}{Z}}+x-{\frac {g{\it I_{yy}}\,X\theta}{ZM}} - I_{yy}\left({\frac {Xqv}{ZM}}+{\frac {qw}{M}}\right) \sin( \theta) \\
       &+I_{yy}\left({\frac {Xqw}{ZM}} -{\frac {qv}{M}} -{\frac {{X}^{2}+ {Z}^{2}}{ZMm}} \right) \cos ( \theta ) \\
\psi_3={}&{q}^{2}-2\,{\frac {M\theta}{I_{yy}}} \\
  \end{split}
\end{equation*}

Given a symbolic initial point $\vx_0=(v_0, w_0, x_0, z_0, \theta_0,q_0,$ $d_1^0, d_2^0)$, we have verified that our invariant cluster
defines the same algebraic variety as defined by the invariants in~\cite{ghorbal2014characterizing} by comparing their Gr{\"o}bner bases. However, our method is much more efficient. Moreover, we also obtained the invariant clusters of higher degrees ($4-6$) quickly. The experimental result is shown in Table~\ref{tbl:benchmarks}. The first column is the degree of the invariants, the second column is the variables to be decided, the third column is the computing time in seconds, and the last column is the number of invariant clusters generated.  As can be seen, in the most complicated case, where the number of the indeterminates reaches up to $3003$, we spent only~$200.9$ seconds to discover an invariant cluster of degree $6$. However, we found that these higher order invariant clusters have the same expressive power as the invariant cluster of degree $3$ in terms of algebraic variety.

\subsection{Looping particle}
Consider a heavy particle on a circular path of radius $r$. The motion of the particle is described by the following differential equation.
\vspace*{-.5\baselineskip}
\begin{gather*}
 \begin{bmatrix}\dot{x} \\ \dot{y} \\ \dot{\omega}\end{bmatrix} = \begin{bmatrix} r\dot{\cos(\theta)} \\ r\dot{\sin (\theta)} \\ -\frac{g\cos (\theta)}{r}\end{bmatrix} = \begin{bmatrix} -r\sin (\theta) \dot{\theta} \\ r\cos (\theta) \dot{\theta} \\ -\frac{g(r\cos (\theta))}{r^2}\end{bmatrix} = \begin{bmatrix} -y\omega \\ x\omega \\ -\frac{gx}{r^2}\end{bmatrix}
 \end{gather*}

Note that the above is a parameterized system with gravity acceleration $g$ and radius $r$ as parameters. Our tool finds the following invariant cluster consisting of a parametric polynomial of degree~$2$: $\{ g(\vu,\vx)=0 \mid g(\vu,\vx) = u_5 x^2 + u_5 y^2 + u_2 \omega^2 + \frac{2 u_2 g}{r^2} y +u_0, \vu\in \R^3 \setminus \{\vect{0}\}\}$. Given an arbitrary point $(x_0,y_0,\omega_0)=(2,0,\omega_0)$, we get the invariant class $\{g(\vu,\vx)=0 \mid (x_0^2 + y_0^2) u_5 + (\omega_0^2 + \frac{2 g}{r^2} y_0) u_2 + u_0 = 0, \vu\in \R^3\setminus \{\vect{0}\}\}$. According to Algorithm~\ref{algo_invclass}, the algebraic variety representing the trajectory originating from $(x_0,y_0,\omega_0)$ is $\{(x,y,\omega)\in \R^3 \mid x^2 + y^2 - x_0^2 - y_0^2 = 0, \omega^2 + \frac{2 g}{r^2} y - \omega_0^2 - \frac{2 g}{r^2} y_0 = 0\}$. The results in \cite{rebiha2015generating} and \cite{sankaranarayanan2004constructing} are special cases of our result when setting $(r,g)=(2,10)$ and  $(r,g,x_0,y_0)=(2,10,2,0)$, respectively. Therefore, our method is more powerful in finding parameterized invariants for parameterized systems. See Table~\ref{tbl:benchmarks} for detailed experimental result.

\subsection{Coupled spring-mass system}\label{coupledmass}
Consider the system
 \begin{gather*}
 \begin{bmatrix}\dot{x_1} \\ \dot{v_1} \\ \dot{x_2} \\ \dot{v_2}\end{bmatrix} = \begin{bmatrix} v_1 \\ -\frac{k_1}{m_1}x_1 - \frac{k_2}{m_1}(x_1-x_2) \\v_2 \\ -\frac{k_2}{m_2}(x_2-x_1)\end{bmatrix}
 \end{gather*}

The model consists of two springs and two weights $w_1, w_2$. One spring, having spring constant $k_1$, is attached to the ceiling and the weight $w_1$ of mass $m_1$ is attached to the lower end of this spring. To the weight $w_1$, a second spring is attached having spring constant $k_2$. To the bottom of this second spring, a weight $w_2$ of mass $m_2$ is attached. $x_1$ and $x_2$ denote the displacements of the center of masses of the weights $w_1$ and $w_2$ from equilibrium, respectively.

In this benchmark experiment, we first tried an instantiated version of the system by using the same parameters as in~\cite{sankaranarayanan2010automatic}: $\frac{k_1}{m_1}=\frac{k_2}{m_2}=k$ and $m_1=5m_2$. The experimental result is presented in Table~\ref{tbl:benchmarks}. We found that the expressive power of the invariant clusters does not increases any more as the degree is greater than $3$ and it took only $0.250$ seconds to compute the invariant cluster of degree $3$. Finally, we perform the computation directly on the fully parameterized system and we get
the following parameterized invariant cluster.

\vspace*{-\baselineskip}
\begin{equation*}
  \begin{split}
    g(\vect{u},\vect{x}) &= u_{8}v_1v_2 + \frac{k_2x_1x_2(m_1u_{8}-2m_2u_{10})}{m_1m_2} + u_{10}v_1^2 +u_{1} \\
    &+\ \frac{1}{2} \frac{v_2^2(k_1m_2u_{8}-k_2m_1u_{8}+k_2m_2u_{8}+2k_2m_2u_{10})}{k_2m_1} \\
    &+\ \frac{1}{2}\frac{(2k_1m_2u_{10}-k_2m_1u_{8}+2k_2m_2u_{10})x_1^2}{m_1m_2} \\
    &+\ \frac{1}{2}\frac{(k_1m_2u_{8}-k_2m_1u_{8}+2k_2m_2u_{10})x_2^2}{m_1m_2}
  \end{split}
\end{equation*}

This invariant cluster enables us to analyze the system properties under different parameter settings.

\subsection{3D-Lotka-Volterra system}
Consider the system
$
 \left[ \dot{x}, \dot{y}, \dot{z} \right] = [ xy - xz, yz - yx, xz - yz ].
$
The experimental result is presented in Table~\ref{tbl:benchmarks}. Here we present only two invariant clusters to show their expressive power: one of degree~$1$ and one of degree~$3$.
$
 \clus_1=\{g_1(\vect{u},\vect{x})= u_3 x + u_3 y + u_3 z + u_4 = 0 \mid \vect{u}\in \R^2\setminus \{\vz\}\},
 \clus_3=\{g_3(\vect{u},\vect{x}) = 0 \mid \vect{u}\in \R^4\setminus \{\vz\}\},
$
where
$
 g_3(\vect{u},\vect{x})= \ u_5 x y z + u_{10}x^3 + 3u_{10}x^2y + 3 u_{10}x^2 z + 3 x y^2   u_{10} + 3 x z^2   u_{10}
 + u_{10}y^3    + 3 u_{10}y^2 z    + 3 u_{10}y z^2    + u_{10}z^3    + u_{16}x^2    + 2 u_{16}x y
 +2 u_{16}x z    + u_{16}y^2 + 2u_{16} yz + u_{16}z^2 + u_{19}x + u_{19}y + u_{19}z    + u_{20}.
$
Invariant cluster $\clus_1$ shows that every trajectory of the system must lie in some plane. $\clus_3$ can overapproximate the trajectories with a much higher precision because it can provide algebraic varieties of dimension $1$.
For example, in \cite{sankaranarayanan2010automatic}, the initial set is given as
$\Init=\{(x,y,z)\in \mathbb{R}^3 \mid \ x^2-1=0$, $y^2-1=0$, $z^2-1=0\}$ and they discovered four invariants.
In our experiment, we spend $0.109$ seconds discovering an invariant cluster. This invariant cluster can overapproximate all the trajectories precisely, say, for $x_0=(1,-1,1)\in \Init$, we get the invariant class $\class(\clus_3,x_0)=\{g_3(\vect{u},\vect{x})=0 \mid -u_5+ 7u_{10}+ u_{16} + u_{19} + u_{20}=0, \vu\in \R^5\setminus\{\vz\}\}$.  By taking the intersection of its basis $xyz + 1=0$ and $x+y+z-1=0$, we immediately obtain a one-dimensional algebraic variety which overapproximates the trajectory precisely.

\subsection{Hybrid controller}\label{hybridexamp}
Consider a hybrid controller consisting of two control modes. The discrete transition diagram of the system is shown in \figurename~\ref{fig:hybridauto} and the vector fields describing the continuous behaviors are as follows:
\begin{align*}
f_1(\vect{x}) &= \begin{bmatrix}y^2+10y+25 \\ 2xy + 10x - 40y -200\end{bmatrix}, \\
f_2(\vect{x}) &= \begin{bmatrix}-y^2-10y-25\\ 8xy+40x-160y-800\end{bmatrix}
\end{align*}
The system starts from some point in $X_0 = \{(x,y) \in \R^2 \mid (x - 9)^2 + (y - 20)^2 \leq 4\}$ and then evolves following the vector field $f_1(\vect{x})$ at location $l_1$ (Switch-On). The value of $x$ keeps increasing until it reaches $35$, then the system switches immediately to location $l_2$ (Switch-Off) without performing any reset operation. At location $l_2$, the system operates following the vector field $f_2(\vect{x})$ and the value of $x$ keeps decreasing. As the value of $x$ drops to $5$, the system switches immediately back to location $l_1$ again. Our objective is to verify that the value of $y$ will never exceed $48$ in both locations.

For the convenience of \SOS programming, we define the unsafe set as $\HUns(l_1)=\HUns(l_2)=\{(x,y)\in \R^2 \mid 48 < y < 60\}$, which is sufficient to prove $y \leq 48$ in locations $l_1$ and $l_2$. According to the theory proposed in Section~\ref{sec:hybridverification}, we first find an invariant cluster for each location, which is composed of a parameterized polynomial, respectively: $g_1(\vu_1,\vx)= -\frac{1}{5} u_{12} x^2+ \frac{1}{10} u_{12} y^2 + 8u_{12} x + u_{12} y + u_{11}$ and $g_2(\vu_2,\vx)= \frac{4}{5} u_{22} x^2 + \frac{1}{10} u_{22} y^2 - 32u_{22} x + u_{22}y+ u_{21}$. In the second phase, we make use of the constraint condition in Proposition~\ref{prop13} to compute a pair of vectors $\vu_1^*$ and $\vu_2^*$. By setting $\gamma_{12}=\gamma_{21}=1$, our tool found a pair of $\vu_1^*=(u_{11},u_{12})=(2.9747,382.14)$ and $\vu_2^*=(u_{21},u_{22})=(2.9747,138.44)$. As shown in \figurename~\ref{fig:hybridinv}, the curves of $g_1(\vu_1^*,\vx)=0$ and $g_2(\vu_2^*,\vx)=0$ form an upper bound for the reachable set in location $l_1$ and $l_2$, respectively, which lie below the unsafe region $y\geq 48$. Therefore, the system is safe.


\begin{figure}[t!]
\begin{subfigure}[b]{0.495\linewidth}
  \centering
  \includegraphics[scale=0.33]{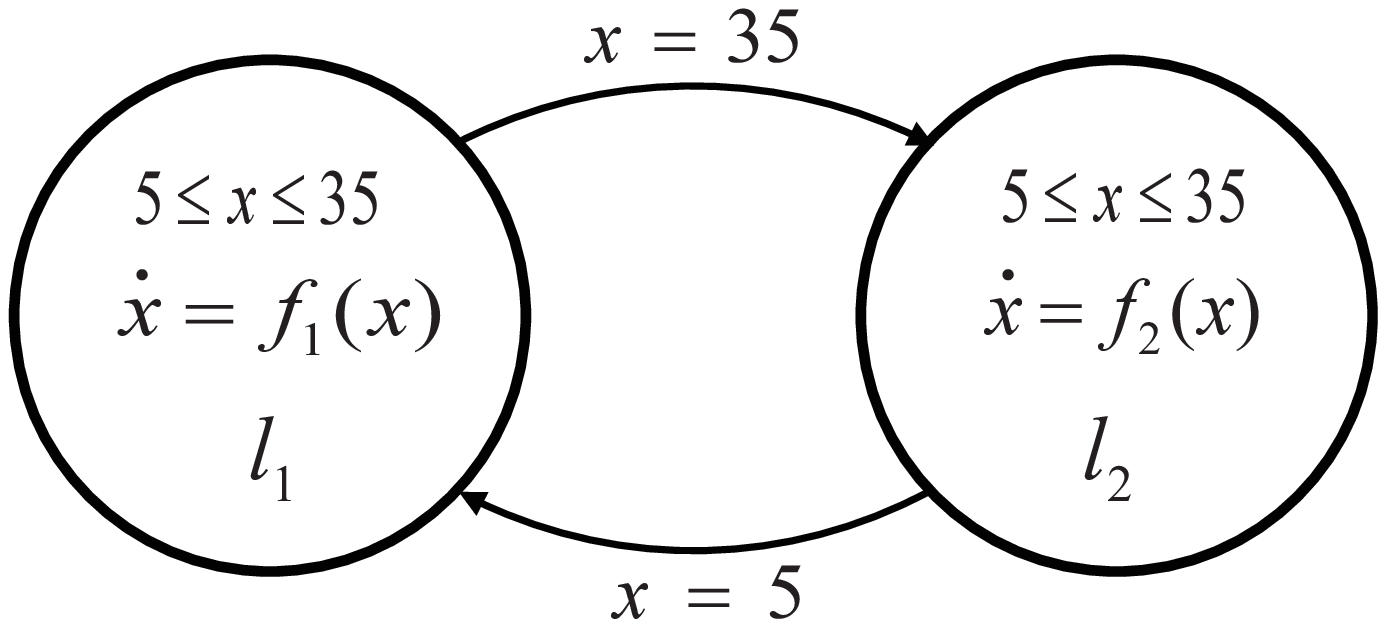}
  \caption{}
  \label{fig:hybridauto}
\end{subfigure}
\begin{subfigure}[b]{0.495\linewidth}
  \centering
  \includegraphics[height=30mm,keepaspectratio]{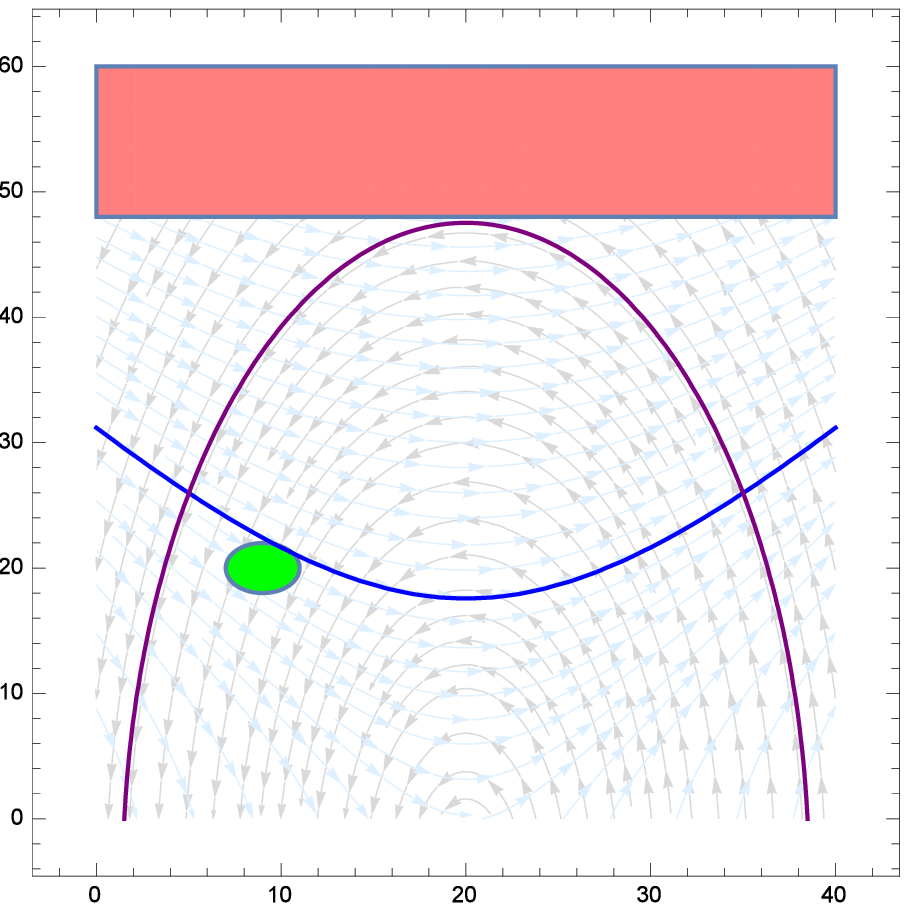}
  \caption{}
  \label{fig:hybridinv}
\end{subfigure}
\caption{(a) Hybrid automaton from Subsection~\ref{hybridexamp}. $x=5$ and $x=35$ are guards for discrete transitions and no reset operation is performed.
(b) Hybrid invariant for the system in Subsection~\ref{hybridexamp}. Solid patch in green: initial set, curve in blue: invariant for $l_1$, curve in purple: invariant for $l_2$, red shadow region on the top: unsafe region.}
\vspace*{-5mm}
\end{figure}

\section{Related Work}\label{sec:relatedwork}

In recent years, many efforts have been made toward generating invariants for hybrid systems. Matringe et al. reduce the invariant generation problem to the computation of the associated eigenspaces by encoding the invariant constraints as symbolic matrices~\cite{matringe2010generating}. Ghorbal et al. use the invariant algebraic set formed by a polynomial and a finite set of its successive Lie derivatives to overapproximate vector flows~\cite{ghorbal2014characterizing}. Both of the aforementioned methods involve minimizing the rank of a symbolic matrix, which is inefficient in dealing with parametric systems. In addition, none of these methods involve how to verify safety properties based on the invariants. Sankaranarayanan discovers invariants based on invariant ideal and pseudo ideal iteration~\cite{sankaranarayanan2010automatic}, but this method is limited to algebraic systems. Tiwari et al. compute invariants for special types of linear and nonlinear systems based on \emph{Syzygy} computation and \Groba theory as well as linear constraint solving~\cite{tiwari2004nonlinear}. Platzer et al. use quantifier elimination to find differential invariants~\cite{platzer2008computing}. The methods based on \Groba and first-order quantifier elimination suffer from the high complexity significantly. Another approach considers barrier certificates based on different inductive conditions~\cite{prajna2004safety,kong2013exponential} which can be solved by \SOS programming efficiently but is limited by the conservative inductive condition. Carbonell et al. generate invariants for linear systems~\cite{rodriguez2005generating}. Some other approaches focusing on different features of systems have also been proposed for constructing inductive invariants~\cite{johnson2013invariant,gulwani2008constraint,sankaranarayanan2004constructing,sassi2014iterative,liu2011computing}.

\section{Conclusion}\label{sec:conclusion}
In this paper, we proposed an approach to automatically generate invariant clusters for semialgebraic hybrid systems. The benefit of invariant clusters is that they can overapproximate trajectories of the system precisely. The invariant clusters can be obtained efficiently by computing the remainder of the Lie derivative of a template polynomial w.r.t. its \Groba and then solving a system of homogeneous polynomial equations obtained from the remainder. Moreover, based on invariant clusters and \emph{SOS} programming, we propose a new method for safety verification of hybrid systems. Experiments show that our approach is effective and efficient for a large class of biological and control systems.

\section*{Acknowledgement}
This research was supported in part by the European Research Council
(ERC) under grant 267989 (QUAREM) and  by the Austrian Science Fund
(FWF) under grants S11402-N23 (RiSE) and Z211-N23 (Wittgenstein
Award).

\bibliographystyle{IEEEtran}
\bibliography{IEEEabrv,myrefs}

\vfill
\pagebreak
\appendix

\section{Appendix}

\subsection{Proof of Proposition~\ref{prop9}}\label{prfofprop9}
\begin{proof}


Let $g(x)\in \R[\vx]$ be a polynomial satisfying $\mathcal{L}_f g\in \ideal{g}$, we first prove that
\begin{equation}\label{eq14}
\forall k \geq 1: \mathcal{L}_f^k g\in \ideal{g}
\end{equation}
 The proof is by induction. Since $g(x)$ satisfies $\mathcal{L}_f g\in \ideal{g}$, we must have $\mathcal{L}_f g = p_1g$ for some $p_1\in \R[x]$. Assume $\forall k\leq m: \mathcal{L}_f^k g\in \ideal{g}$ holds, we only need to prove $\mathcal{L}_f^{m+1} g\in \ideal{g}$ holds. By assumption, we have $\mathcal{L}_f^{m} g\in \ideal{g}$, that is, $\mathcal{L}_f^{m} g = p_{m} g$ for some $p_{m}\in \R[x]$. Then, $\mathcal{L}_f^{m+1} g = \mathcal{L}_f \mathcal{L}_f^{m} g = \mathcal{L}_f p_{m} g = g \mathcal{L}_f p_{m} + p_{m} \mathcal{L}_f g = g \mathcal{L}_f p_{m} + p_{m}p_1 g = g(\mathcal{L}_f p_{m} + p_{m}p_1)$. Therefore, the formula~\eqref{eq14} holds.

Second, we prove that $g(x)=0$ is an invariant, i.e. for any trajectory $x(t)$,
\begin{equation}\label{eq13}
 g(x(0)) = 0 \implies \forall t\geq 0: g(x(t)) = 0
\end{equation}
By the formula~\eqref{eq14}, we have $\frac{d^m g(x(t))}{dt^m}|_{t=\tau} = \mathcal{L}_f^m g|_{t=\tau} = p_m g(x(\tau))$ for all $m \geq 1$, then $g(x(0))=0\implies \frac{d^m g(x(t))}{dt^m}|_{t=0}=0$ for all $m \geq 1$. Consider the Taylor expansion of $g(x(t))$ at $t=\tau$,
\begin{equation}\label{eq16}
  g(x(t)) = g(x(\tau)) + \sum_{m=1}^{\infty}\frac{1}{m!}\frac{d^m g(x(t))}{dt^m}|_{t=\tau}(t-\tau)^m
\end{equation}
Hence, the formula~\eqref{eq13} holds.

Next, we prove that for any trajectory $x(t)$,
\begin{equation}\label{eq15}
 g(x(0)) \neq 0 \implies \forall t\geq 0: g(x(t)) \neq 0
\end{equation}
then, according to the continuity of $g(x(t))$, we can assert that $g(x)\sim 0$ is an invariant for each $\sim \in \{<, \leq, =, \geq, >\}$.

The proof is by contradiction. Given an arbitrary trajectory $x(t)$, assume $g(x(\tau))=0$ for some $\tau>0$ and $g(x(t))>0$ (or $g(x(t))<0$) for all $t\in [0,\tau)$. Since $\frac{d^m g(x(t))}{dt^m}|_{t=\tau} = \mathcal{L}_f^m g|_{t=\tau} = p_m g(x(\tau))=0$ for all $m\geq 1$, then according to the Taylor expansion~\eqref{eq16}, there must exist a $\delta\in \R_{>0}$ such that $g(x(t))=0$ for all $t\in [\tau-\delta, \tau+\delta]$, which contradicts the assumption that $g(x(t))>0$ (or $g(x(t))<0$) for all $t\in [0,\tau)$. Thus, the proposition holds.

\end{proof}

\subsection{Proof of Theorem~\ref{thrm7}}\label{thrm7proof}
\begin{proof}
\ref{item1}. By Definition~\ref{def_inst}, for all $g(\vect{u},\vect{x})\in D_g$, $g(\vect{u},\vect{x})=0$ is an invariant for the trajectory originating from $\vect{x}_0$, which means $\pi_{\vect{x}_0}\subseteq \{\vect{x}\in \R^n \mid g(\vect{u},\vect{x})=0\}$. Hence, $\pi_{\vect{x}_0} \subseteq \V(D_g)$ holds.

\ref{item2}. Since the dimension of the hyperplane $g(\vect{u},\vect{x}_0)=0$ is $m$, there must exist $m$ vectors $\{\vect{u}_1,\dots,\vect{u}_{m}\}$ such that for all $\vect{u}$ satisfying $g(\vect{u},\vect{x}_0)=0$, there exist scalars $c_1,\dots,c_{m}$, not all zero, such that $\vect{u}=\sum_{i=1}^{m}c_i \vect{u}_i$. Therefore, for every $g(\vect{u},\vect{x})\in D_g$, we have $g(\vect{u},\vect{x})=g(\sum_{i=1}^{m}c_i \vect{u}_i, \vect{x})=\sum_{i=1}^{m}c_i g(\vect{u}_i,\vect{x})$, hence, $g(\vect{u},\vect{x})\in \ideal{g(\vect{u}_1,\vect{x}),\dots,g(\vect{u}_{m},\vect{x})}$. For $B=\{g(\vect{u}_1,\vect{x}),\dots,g(\vect{u}_{m},\vect{x})\}$, we have $\ideal{D_g} \subseteq \ideal{B}$. The converse is trivial. Therefore, $\ideal{D_g}\ =\ \ideal{B}$.

\end{proof}

\subsection{Principle of \SOS programming}\label{sosprogramming}
The principle of \SOS programming is based on the fact that a polynomial of degree~$2k$ can be written as a sum-of-squares (\SOS) $P(\vx) = \sum q_i(\vx)^2$ for some polynomials $q_i(\vx)$ of degree~$k$ if and only if $P(\vx)$ has a positive semidefinite quadratic form, i.e. $P(\vx) = \vect{v}(\vx)M\vect{v}(\vx)^T$, where $\vect{v}(\vx)$ is a vector of monomials with respect to $x$ of degree~$k$ or less and $M$ is a real symmetric positive semidefinite matrix with the coefficients of $P(\vx)$ as its entries. Therefore, the problem of finding a \SOS polynomial $P(\vx)$ can be converted to the problem of solving a linear matrix inequality (\emph{LMI}) $M\succeq 0$~\cite{boyd1994linear}, which can be solved by semidefinite programming~\cite{parrilo2003semidefinite}. Currently, there exists an efficient implementation, named \emph{SOSTOOLS}~\cite{prajna2005sostools}, for \SOS programming and our implementation is based on this tool. For details on \SOS programming we refer the reader to the literature.

\subsection{Proof of Theorem~\ref{thrm8}}\label{proofofthrm8}
\begin{proof}
When considering $g(\vect{u},\vect{x})=0$ as a parameterized hyperplane over $\vect{u}$, since there is no constant term, the normal vector of the hyperplane is $(\psi_1(\vect{x}_i),$ $\dots,\psi_K(\vect{x}_i))$. Therefore, if $\vect{x}_1$ and $\vect{x}_2$ are in the same trajectory, they must correspond to the same hyperplane, namely, the normal vectors of $g(\vect{u},\vect{x}_1)$ and $g(\vect{u},\vect{x}_2)$ are either all equal to $\vect{0}$, which indicates the formula \eqref{zeronormalvec} holds, or are proportional to one another, which means the formula~\eqref{normalvec} holds. Moveover, if $\psi_i(\vect{x})\equiv 1$, then formula~\eqref{normalvec1} follows from \eqref{normalvec} immediately.

\end{proof}

\subsection{Proof of Theorem~\ref{thrm2}}\label{proofofthrm2}
\begin{proof}
By Theorem~\ref{citedthrm1}, the identity~\eqref{eq12} holds if and only if there exists no real solution for the following system $P$ of polynomial equations and inequalities
\begin{equation}\label{polyidentity}
\begin{split}
 & \psi_k(\vx_1)-\psi_k(\vx_2) = 0,\ k=1 \dots K\\
 & p_{i_1}(\vx_1) = 0,\ p_{i_2}(\vx_2)=0,\\
 &\qquad i_1=1, \dots, l_1, i_2=l_1+1, \dots l_1 + l_2\\
 & q_{j_1}(\vx_1) \geq 0,\ q_{j_2}(\vx_2) \geq 0,\\
 &\qquad j_1=1, \dots, m_1,\ j_2=m_1+1 \dots m_1+m_2\\
 & r_{k_1}(\vx_1) > 0,\ r_{k_2}(\vx_2) > 0,\\
 &\qquad k_1=1, \dots, n_1,\ k_2=n_1+1 \dots n_1+n_2.
\end{split}
\end{equation}
That formula~\eqref{polyidentity} has no solution indicates that there exists no $(\vx_1,\vx_2)\in \Init\times \Uns$ such that $\psi_i(\vx_1)=\psi_i(\vx_2)$ for all $i=1 \dots K$. By the assumption $K\geq 2$, we have $\class(C,x_i) \neq \emptyset, i=1,2$ for any $(\vect{x}_1,\vect{x}_2)\in \Init\times\Uns$. By Theorem~\ref{thrm8}, the system is safe.

\end{proof}

\subsection{Proof of Proposition~\ref{prop10}}\label{proofofprop10}
\begin{proof}
By Proposition~\ref{prop9}, $g(\vect{u},\vect{x})\geq 0$ is an invariant of $\sys$ for any $\vect{u}\in \R^K$. Suppose $g(\vect{u}^*,\vect{x})$ satisfies the constraints~\eqref{cond1} and \eqref{cond2} for some $\vect{u}^*\in\R^K$ and $\vect{x}(t)$ is an arbitrary trajectory of $\sys$ such that $\vect{x}(0)\in \Init$, we must have $g(\vect{u}^*,\vect{x}(t)) \geq 0$ for all $t\geq 0$, which means $\vect{x}(t)$ cannot reach $\Uns$. Therefore, the system is safe.

\end{proof}

\subsection{Proof of Proposition~\ref{prop11}}\label{proofofprop11}
\begin{proof}
Since the polynomial \eqref{cond21} is an \SOS, then $\forall \vect{x}\in \R^n: -g(\vect{u}^*,\vect{x}) - \vect{\mu}_1\cdot \vect{I} \geq 0$, i.e. $\forall \vect{x}\in \R^n: -g(\vect{u}^*,\vect{x}) \geq \vect{\mu}_1\cdot \vect{I}$. Since $\forall \vect{x}\in \Init: \vect{I} \succeq \vz$ and $\vect{\mu}_1$ is an \SOS polynomial vector, which implies $\forall \vx\in \R^n:\vect{\mu}_1\succeq \vz$, then $\forall \vx\in\Init: -g(\vect{u}^*,\vect{x}) \geq \vect{\mu}_1\cdot \vect{I} \geq 0$. Similarly, we can derive from the \SOS polynomial~\ref{cond22} that $\forall \vx\in \Uns: -g(\vect{u}^*,\vect{x}) < 0$. By Proposition~\ref{prop10}, we can conclude that the system is safe.

\end{proof}

\subsection{Proof of Proposition~\ref{prop12}}\label{proofprop12}
\begin{proof}
Assume $S_u$ is the set that satisfies the formulae~\eqref{cond31}, \eqref{cond32} and \eqref{cond33}, by Proposition~\ref{prop9}, every $g_l(\vu_l^*,\vx)\geq 0$ is an invariant at location $l$. To prove this proposition, it is sufficient to prove that given any trajectory, say $\pi$, of the system $\mathcal{H}$, it cannot reach an unsafe state. Suppose the infinite time interval $\mathbb{R}_{\geq 0}$ associated with $\pi$ is divided into an infinite sequence of continuous time subintervals, i.e., $\mathbb{R}_{\geq 0} = \bigcup_{n=0}^\infty I_n$, where $I_n = \{t\in \mathbb{R}_{\geq 0}| t_n \leq t \leq t_{n+1, t_n<t_{n+1}}\}$ is the time interval that the system spent at location $\rho(I_n)$ (where $\rho(I_n)$ returns the location corresponding to $I_n$ associated to $\pi$), we define the trajectory as $\pi = \{(\rho(I_n),\vx(t))| t \in I_n, n \in \mathds{N}\}$, where $\vx(t_0) \in \HInit(\rho(I_0))$. Then, our objective is to prove the following assertion:
\begin{equation}\label{assertion1}
\forall n \in \mathds{N}: \forall t\in I_n:g_{\rho(I_n)}(\vu_{\rho(I_n)}^*,\vx(t))\geq 0.
\end{equation}
The basic proof idea is by induction.

\noindent \emph{Basis:}
$n=0$. Since $g_{\rho(I_0)}(\vu_{\rho(I_0)}^*,\vx)\geq 0$ is an invariant, it is obvious that
\[\forall t\in I_0:g_{\rho(I_0)}(\vu_{\rho(I_0)}^*,\vx(t))\geq 0.\]

\noindent \emph{Induction:}
$n=k$. Assume for some $k$, \[\forall j \in [0,k]:\forall t\in I_j: g_{\rho(I_j)}(\vu_{\rho(I_j)}^*,\vx(t))\geq 0.\]
we mean to prove that \[\forall t\in I_{k+1}:g_{\rho(I_{k+1})}(\vu_{\rho(I_{k+1})}^*,\vx(t))\geq 0.\]

\noindent Case $1$. (Discrete Transition) By the inductive assumption, we know that
\[\forall t\in I_k: g_{\rho(I_k)}(\vu_{\rho(I_k)}^*,\vx(t))\geq 0.\]
hence \[\forall t\in I_k:\vx(t)\in G(\rho(I_k),\rho(I_{k+1})) \implies g_{\rho(I_k)}(\vu_{\rho(I_k)}^*,\vx(t)) \geq 0\]
According to condition~\eqref{cond32}, $ \forall \vx'(t) \in R((l,l'),\vx):g_{\rho(I_{k+1})}(\vu_{\rho(I_{k+1})}^*,\vx'(t)) \geq 0$ holds.

\noindent Case 2. (Continuous Transition) \\
Based on Case $1$ and the fact that $g_{\rho(I_{k+1})}(\vu_{\rho(I_{k+1})}^*, \vx)\geq 0$ is an invariant at $\rho(I_{k+1})$, we can conclude that $\forall t \in I_{k+1}:g_{\rho(I_{k+1})}(\vu_{\rho(I_{k+1})}^*, \vx(t))\geq 0$.

By induction, we know that the assertion~\eqref{assertion1} holds. Therefore, the system is safe.
\end{proof}

\subsection{Invariant clusters of Example~\ref{example8}}\label{append_exam2}
\begin{equation*}
  \begin{split}
  &\clus_1=\{u_2(x-y)=0\mid u_2\in \R\},\ \clus_2=\{u_2(x+y)=0\mid u_2\in \R\},\\
  &\clus_3=\{u_3(x+z)=0\mid u_3\in \R\},\ \clus_4=\{u_3(x-z)=0\mid u_3\in \R\},\\
  &\clus_5=\{u_3(y+z)=0\mid u_3\in \R\},\ \clus_6=\{u_3(y-z)=0\mid u_3\in \R\}.\\
  \end{split}
\end{equation*}

\subsection{Overapproximating a convex compact set}\label{overapprox}
In fact, overapproximating a compact semialgebraic set by one or multiple ellipsoids is reasonable. If a compact semialgebraic set is overapproximated by an ellipsoid $(\vx-\vect{v})^TA(\vx-\vect{v})\leq 1$, where $A$ is a positive definite matrix and $x, v$ are vectors, then we can use $(x-v)^TA(x-v) = 1$ equivalently instead of $(\vx-\vect{v})^TA(\vx-\vect{v})\leq 1$ for the verification, the reason is that any trajectory originating from the interior of the ellipsoid must pass through some point on the surface of the ellipsoid.

\subsection{Additional information for Example~\ref{example9}}\label{verifydata}
\begin{equation*}
  \begin{split}
  \alpha_1 =& -0.027854 x_1^2 - 0.048663y_1^2 - 0.014007x_2^2 - 0.031091 y_2^2\\
	              & - 0.019017\\
  \alpha_2 =& -0.0227x_1^2 - 0.027808 y_1^2 - 0.034068x_2^2 - 0.030096y_2^2 \\
                 &- 0.00016071\\
  \alpha_3 =& -1.8255x_1^2 + 1.8251y_1^2 + 1.8211x_2^2 - 1.8282y_2^2 + 0.014587
  \end{split}
\end{equation*}
\begin{equation*}
\begin{split}
  \Prog = &\ 1.8533 x_1^4 - 3.5741 x_1^2 y_1^2 - 3.6098 x_1^2 x_2^2 + 3.7075 x_1^2 y_2^2\\
  &+ 1.8738 y_1^4 + 3.688 y_1^2 x_2^2 - 3.5944 y_1^2 y_2^2 + 1.8551 x_2^4\\
  &- 3.5851 x_2^2 y_2^2 + 1.8583 y_2^4 + 0.83561 x_1^3 - 0.94703 x_1^2 y_1\\
  &- 0.49941 x_1^2 x_2 - 0.74911 x_1^2 y_2 + 1.4599 x_1 y_1^2 + 0.4202 x_1 x_2^2\\
  &+ 0.93272 x_1 y_2^2 - 1.6545 y_1^3 - 0.61178 y_1^2 x_2 - 0.91767 y_1^2 y_2\\
  &- 0.47623 y_1 x_2^2 - 1.0571 y_1 y_2^2 - 0.7495 x_2^3 - 1.1242 x_2^2 y_2\\
  &- 0.66212 x_2 y_2^2 - 0.99318 y_2^3 + 23.1976 x_1^2 + 35.9056 y_1^2\\
  &+ 20.5634x_2^2 + 27.7403 y_2^2 + 0.5705 x_1 - 0.64657 y_1\\
  & - 0.0035356 x_2 - 0.0053034y_2 + 9.8186
\end{split}
\end{equation*}

\subsection{Invariant Cluster for Subsection~\ref{motionairplane}}\label{invairplane}

\begin{equation*}
\begin{split}
g_9(\vect{u},\vect{x})=& \left(-{\frac {mqw{d_2}}{Z}}+{\frac {mMz}{{I_{yy}}\,Z}}+{\frac {X{d_2}}{Z}}+d_1+{\frac {mg\theta\,d_1^{2}}{Z}}\right.\\
&\left. +{\frac {mg\theta\,d_2^{2}}{Z}}+{\frac {mqv{d_1}}{Z}} \right) {u}_{5} + \left( {\frac {{I_{yy}}\,Xqw{d_2}}{ZM}}-{\frac {{I_{yy}}\,qw{d_1}}{M}}\right. \\
&-{\frac {Xz}{Z}}+x-{\frac {( {I_{yy}}\,{X}^{2}+{I_{yy}}\,{Z}^{2}) {d_2}}{ZMm}}-{\frac {g{I_{yy}}\,X\theta\,d_1^{2}}{ZM}}\\
&\left. -{\frac {g{I_{yy}}\,X\theta\,d_2^{2}}{ZM}} -{\frac {{I_{yy}}\,qv{d_2}}{M}}-{\frac {{I_{yy}}\,Xqv{d_1}}{ZM}} \right) {u}_{85}\\
&+ \left( d_1^{2}+d_2^{2}\right) {u}_{{8}}+ \left( {q}^{2}-2\,{\frac {M\theta}{{I_{yy}}}} \right) {u}_{{17}}+ {u}_{1}\\
\end{split}
\end{equation*}

%
%

\end{document}